\documentclass{amsart}
\usepackage[nohug]{diagrams}
\usepackage{amscd}
\usepackage[mathscr]{eucal}
\usepackage{amssymb, amsmath, amsthm, braket}
\usepackage{amsfonts}
\usepackage{stmaryrd}
\usepackage{latexsym}
\usepackage{tabularx,array}
\usepackage[all]{xy}
\usepackage{hyperref}
\usepackage{bbm}
\usepackage{latexsym}

\newfont{\msam}{msam10}

\newtheorem{theorem}[]{Theorem}
\newtheorem{proposition}[]{Proposition}
\newtheorem{corollary}[]{Corollary}
\newtheorem{lemma}[]{Lemma}

\newtheorem*{thm1}{Theorem}
\theoremstyle{definition}
\newtheorem{definition}[]{Definition}
\newtheorem{remark}[]{Remark}
\newtheorem{example}[]{Example}

\def\inc{\iota}
\let\nc\newcommand

\def\bthm{\begin{theorem}}
\def\ethm{\end{theorem}}
\def\blemma{\begin{lemma}}
\def\elemma{\end{lemma}}
\def\bproof{\begin{proof}}
\def\eproof{\end{proof}}
\def\bprop{\begin{proposition}}
\def\eprop{\end{proposition}}
\def\bcor{\begin{corollary}}
\def\ecor{\end{corollary}}
\nc{\la}{\label}
\def\Z{\mathbb{Z}}
\def\N{\mathbb{N}}

\def\c{\mathbb{C}}

\def\L {\mathbb{L}}

\def\Com{\mathtt{Com}}

\def\Alg{\mathtt{Alg}}

\def\LAlg{\mathtt{Lie\,Alg}}
\def\DGL{\mathtt{DGLA}}
\def\DGC{\mathtt{DGC}}
\def\cDGC{\mathtt{DGCC}}

\def\cAlg{\mathtt{Comm\,Alg}}

\def\Sets{\mathtt{Sets}}
\def\DGA{\mathtt{DGA}}

\def\cDGA{\mathtt{DGCA}}

\def\C{\mathcal{C}}

\def\Ho{{\mathtt{Ho}}}
\def\mfa{\mathfrak{a}}
\nc{\Ob}{{\rm Ob}}
\nc{\Hom}{{\rm{Hom}}}
\nc{\Homcont}{{\mathcal{H}om}}
\nc{\HOM}{\underline{\rm{Hom}}}
\nc{\DER}{\underline{\rm{Der}}}
\nc{\END}{\underline{\rm{End}}}
\nc{\bSym}{\mathbf{Sym}}
\nc{\Ext}{{\rm{Ext}}}
\nc{\Rep}{{\rm{Rep}}}
\nc{\DRep}{{\rm{DRep}}}
\nc{\NCRep}{\widetilde{\rm{Rep}}}
\nc{\RAct}{{\rm{RAct}}}
\nc{\bs}{\backslash}
\nc{\ob}{{\tt{Obs}}}
\nc{\CE}{\mathcal{C}}
\nc{\TP}{{T\!P}}
\nc{\nn}{{{\natural} {\natural}}}
\nc{\n}{{{\natural}}}
\nc{\A}{\mathbb A}
\nc{\Sb}{\mathbb S}
\nc{\B}{{\mathrm{B}}}
\nc{\Ba}{\overline{\mathrm{B}}}
\nc{\bC}{\overline{C}}
\nc{\bOmega}{\boldsymbol{\Omega}}
\nc{\bB}{\boldsymbol{B}}
\nc{\EXT}{\underline{\rm{Ext}}}
\nc{\TOR}{\underline{\rm{Tor}}}
\def\H{\mathrm H}

\def\HC{\mathrm{HC}}

\def\rHC{\overline{\mathrm{HC}}}

\nc{\End}{{\rm{End}}}
\nc{\GL}{{\rm{GL}}}
\nc{\gl}{{\mathfrak{gl}}}
\nc{\rgl}{\overline{{\mathfrak{gl}}}}
\nc{\g}{{\mathfrak{g}}}
\nc{\h}{{\mathfrak{h}}}
\nc{\PGL}{{\rm{PGL}}}
\nc{\SL}{{\rm{SL}}}
\nc{\sll}{\mathfrak{sl}}
\nc{\cn}{ \mbox{\rm c\^{o}ne} }
\nc{\PSL}{{\rm{PSL}}}
\nc{\ad}{{\rm{ad}}}
\nc{\Ad}{{\rm{Ad}}}
\nc{\dlim}{\varinjlim}
\nc{\plim}{\varprojlim}
\nc{\colim}{{\tt{colim}}}

\newcommand{\Spec}{{\rm{Spec}}}

\newcommand{\Sym}{{\rm{Sym}}}

\newcommand{\id}{{\rm{Id}}}

\newcommand{\Tr}{{\rm{Tr}}}

\newcommand{\Ker}{{\rm{Ker}}}

\newcommand{\twopartdef}[4]
{
	\left\{
		\begin{array}{ll}
			#1 & \mbox{if } #2 \\
			#3 & \mbox{if } #4
		\end{array}
	\right.
}
\newcommand{\into}{\,\hookrightarrow\,}

\newcommand{\onto}{\,\twoheadrightarrow\,}
\newcommand{\sonto}{\,\stackrel{\sim}{\twoheadrightarrow}\,}
\def\cb{\boldsymbol{\Omega}}

\def\bs{\backslash}

\newcommand{\rar}{\xrightarrow{}}

\nc{\env}{\mathrm{End}(V)}
\nc{\FT}{\mathcal{C}}

\numberwithin{equation}{section}
\numberwithin{theorem}{section}
\numberwithin{lemma}{section}
\numberwithin{proposition}{section}
\numberwithin{corollary}{section}
\numberwithin{example}{section}
\numberwithin{remark}{section}

\def\dr{\mathrm{DR}_{\bullet}}
\def\drm{\mathrm{DR}^{\bullet}}
\newcommand{\dual}{\text{!`}}
\newcommand{\Sh}{\mathrm{Sh}}
\newcommand{\ab}{\mathrm{ab}}

\newcommand{\TTr}{\overline{\Tr}}
\newcommand{\sym}{\operatorname{sym}}
\newcommand{\can}{\mathrm{can}}
\newcommand{\PBT}{\mathfrak{T}}
\newcommand{\clLPBT}{\langle\mathfrak{T}\rangle}
\newcommand{\twbs}{\mathbf{s}}

\newcommand{\twDR}{\mathsf{DR}}

\def\arbreBA{\vcenter{\xymatrix@R=2pt@C=2pt{
&&&&\\
&&&*{}\ar@{-}[ul] & \\
&&*{}\ar@{-}[uurr] \ar@{-}[uull] \ar@{-}[d]     &&\\
&&&&
}}}

\def\arbreAB{\vcenter{\xymatrix@R=2pt@C=2pt{
&&&&\\
&*{}\ar@{-}[ur] &&& \\
&&*{}\ar@{-}[uurr] \ar@{-}[uull] \ar@{-}[d]     &&\\
&&&&
}}}

\def\arbreABC{\vcenter{\xymatrix@R=1pt@C=1pt{
&&&&&&\\
&*{}\ar@{-}[ur] &&&&& \\
&&*{}\ar@{-}[uurr] &&&&\\
&&&*{}\ar@{-}[uuurrr] \ar@{-}[uuulll] \ar@{-}[d] &&&\\
&&&&&&
}}}

\def\arbreBAC{\vcenter{\xymatrix@R=1pt@C=1pt{
&&&&&&\\
&&&*{}\ar@{-}[ul] &&& \\
&&*{}\ar@{-}[uurr] &&&&\\
&&&*{}\ar@{-}[uuurrr] \ar@{-}[uuulll] \ar@{-}[d] &&&\\
&&&&&&
}}}

\def\arbreACB{\vcenter{\xymatrix@R=1pt@C=1pt{
&&&&&&\\
&*{}\ar@{-}[ur] &&&&& \\
&&&&*{}\ar@{-}[uull] &&\\
&&&*{}\ar@{-}[uuurrr] \ar@{-}[uuulll] \ar@{-}[d] &&&\\
&&&&&&
}}}

\def\arbreBCA{\vcenter{\xymatrix@R=1pt@C=1pt{
&&&&&&\\
&&&&&*{}\ar@{-}[ul] & \\
&&*{}\ar@{-}[uurr] &&&&\\
&&&*{}\ar@{-}[uuurrr] \ar@{-}[uuulll] \ar@{-}[d] &&&\\
&&&&&&
}}}

\def\arbreCAB{\vcenter{\xymatrix@R=1pt@C=1pt{
&&&&&&\\
&&&*{}\ar@{-}[ur] &&& \\
&&&&*{}\ar@{-}[uull] &&\\
&&&*{}\ar@{-}[uuurrr] \ar@{-}[uuulll] \ar@{-}[d] &&&\\
&&&&&&
}}}

\def\arbreCBA{\vcenter{\xymatrix@R=1pt@C=1pt{
&&&&&&\\
&&&&&*{}\ar@{-}[ul] & \\
&&&&*{}\ar@{-}[uull] &&\\
&&&*{}\ar@{-}[uuurrr] \ar@{-}[uuulll] \ar@{-}[d] &&&\\
&&&&&&
}}}

\def\arbreACA{\vcenter{\xymatrix@R=1pt@C=1pt{
&&&&&&\\
&*{}\ar@{-}[ur] &&&&*{}\ar@{-}[ul] & \\
&&&&&&\\
&&&*{}\ar@{-}[uuurrr] \ar@{-}[uuulll] \ar@{-}[d] &&&\\
&&&&&&
}}}

\date{May 19, 2015}

\title{Chern-Simons Forms and Higher Character Maps of\\ Lie Representations}

\author{Y. Berest}
\address{Department of Mathematics,
 Cornell University, Ithaca, NY 14853-4201, USA}
\email{berest@math.cornell.edu}
\author{G. Felder}
\address{Departement Mathematik,
Eidgenossische TH Z\"urich,
8092 Z\"urich, Switzerland}
\email{giovanni.felder@math.ethz.ch}
\author{S. Patotski}
\address{Department of Mathematics,
 Cornell University, Ithaca, NY 14853-4201, USA}
\email{apatotski@math.cornell.edu}
\author{A. C. Ramadoss}
\address{Department of Mathematics, Indiana University, Bloomington, IN 47405, USA}
\email{ajcramad@indiana.edu}
\author{T. Willwacher}
\address{Institut f\"{u}r Mathematik, Universit\"{a}t Z\"{u}rich, 8057 Z\"{u}rich, Switzerland}
\email{thomas.willwacher@math.uzh.ch}

\begin{document}
\begin{abstract}
This paper is a sequel to~\cite{BFPRW}, where we study the derived representation
scheme $ \DRep_{\g}(\mfa) $ parametrizing the representations of a Lie algebra $ \mfa $ 
in a reductive Lie algebra $ \g $. In \cite{BFPRW}, we constructed two canonical maps 
$\, \Tr_{\g}(\mfa):\, \HC_{\bullet}^{(r)}(\mfa) \to \H_{\bullet}[\DRep_{\g}(\mfa)]^G $
and $ \Phi_{\g}(\mfa):\,\H_{\bullet}[\DRep_{\g}(\mfa)]^G  \to \H_{\bullet}[\DRep_{\h}(\mfa)]^{\mathbb W}  $ called the Drinfeld trace and the derived Harish-Chandra homomorphism,
respectively. In this paper, we give an explicit formula for the Drinfeld trace in terms of 
Chern-Simons forms. As a consequence, we show that, if $ \mfa $ is an abelian Lie algebra,
the composite map $ \Phi_{\g}(\mfa) \circ \Tr_{\g}(\mfa) $ is given by a canonical
differential operator defined on differential forms on $ A = \Sym(\mfa) $ and depending 
only on the Cartan data $ (\h, {\mathbb W}, P) $, where $ P \in \Sym(\h^*)^{\mathbb W} $. 
We prove a combinatorial formula for this operator that plays an important role in 
the study of derived commuting schemes in \cite{BFPRW}.
\end{abstract}
\maketitle
\tableofcontents

\section{Introduction and Motivation}

Let $ \g $ be a finite-dimensional reductive Lie algebra defined over a field $k$
of characteristic zero. For an arbitrary Lie algebra $ \mfa $, the set
of all representations (i.e., Lie algebra homomorphisms) of $ \mfa $ in $ \g $ 
has a natural structure of an affine $k$-scheme called the representation scheme 
$ \Rep_{\g}(\mfa) $. In \cite{BFPRW}, we constructed a derived version of 
 $ \Rep_{\g}(\mfa) $ by extending the representation functor 
$ \Rep_{\g} $ to the category of differential graded (DG) Lie algebras and 
deriving it in the sense of non-abelian homological algebra \cite{Q1}.
The corresponding derived  scheme $ \DRep_{\g}(\mfa) $ is represented by a 
commutative DG algebra which (to simplify the notation) we also denote 
by $ \DRep_{\g}(\mfa) $. The DG algebra $ \DRep_{\g}(\mfa) $ is well defined up 
to homotopy; its homology $ \H_\bullet[\DRep_{\g}(\mfa)] $ depends only on $ \mfa $ and $ \g $, with 
$ \H_0[\DRep_{\g}(\mfa)] $ being canonically isomorphic to the coordinate ring 
$ k[\Rep_{\g}(\mfa)] $ of $ \Rep_{\g}(\mfa) $. Following \cite{BFPRW}, 
we call $ \H_\bullet[\DRep_{\g}(\mfa)] $ the {\it representation homology} of
$ \mfa $ in $ \g $ and denote it by $ \H_\bullet(\mfa, \g) $. The algebraic group
$G$ associated with $ \g $ acts naturally on $ \Rep_{\g} $ via the adjoint
representation. This action extends, by functoriality, to representation 
homology, and we define $ \H_\bullet(\mfa, \g)^G $ to be the $G$-invariant part 
of $ \H_\bullet(\mfa, \g) $ (see Section~\ref{DREP} for a brief review of the
above constructions).

In general, computing $ \H(\mfa, \g) $ and $ \H_\bullet(\mfa, \g)^G $ is a difficult
problem: an explicit presentation for these algebras is known only in a few nontrivial cases. 
It is therefore natural to look for some canonical maps relating representation homology to more
accessible invariants. In \cite{BFPRW}, we defined two such maps\footnote{We will review 
the construction of these maps in Section~\ref{SecCSCE} below.}:
\begin{equation}
\la{eq1}
\Tr_{\g}(\mfa):\, \HC_{\bullet}^{(r)}(\mfa) \to \H_{\bullet}(\mfa, \g)^G 
\end{equation}
\begin{equation}
\la{eq2}
\Phi_{\g}(\mfa):\, \H_{\bullet}(\mfa, \g)^G \to \H_{\bullet}(\mfa, \h)^{\mathbb W}
\end{equation}
which we called the {\it Drinfeld trace} and the {\it derived Harish Chandra homomorphism},
respectively. The Drinfeld trace is defined on the $r$-th Hodge component of the cyclic
homology of the universal enveloping algebra $ {\mathcal U} \mfa $ of the Lie algebra $ \mfa $ 
and depends on the choice of a $G$-invariant polynomial $ P \in \Sym^r(\g^*)^G $
on the Lie algebra $ \g$: it should be thought of as a derived extension
of the classical character maps for Lie representations. The Harish Chandra 
homomorphism $ \Phi_{\g}(\mfa) $ is a graded algebra homomorphism that extends to 
representation homology the natural restriction map 
$\, k[\Rep_{\g}(\mfa)]^G \to k[\Rep_{\h}(\mfa)]^{\mathbb W} $, where $ \h \subset \g $
is a Cartan subalgebra of $ \g $ and $ {\mathbb W} $ is the associated Weyl group.

The maps \eqref{eq1} and \eqref{eq2} are particularly interesting when $ \mfa $ is a two-dimensional abelian Lie algebra. In this case, $ \DRep_{\g}(\mfa) $ represents the derived  commuting scheme of the Lie algebra $ \g $, a higher homological extension of the 
classical commuting scheme $ \Rep_{\g}(\mfa) = \{(x,y) \in \g \times \g \,:\,[x,y] = 0\} $
parametrizing the pairs of commuting elements in $ \g $. It turns out 
that if $ \mfa $ is graded, with generators having opposite parities, then 
$ \H_{\bullet}(\mfa, \g)^G $ is a free graded commutative algebra 
generated by the (images of) Drinfeld traces \eqref{eq1} corresponding to free polynomial 
generators $ \{P_1, \ldots, P_l\} $ of the invariant algebra $ \Sym(\g^*)^G $. (As shown in \cite{BFPRW}, this result is equivalent to the strong Macdonald conjecture proved in \cite{FGT}.) 
On the other hand, when both generators of $ \mfa $ are even (e.g., have homological degree $0$), it is conjectured in \cite{BFPRW} (with some evidence provided) that the Harish Chandra 
homomorphism $ \Phi_{\g}(\mfa) $ is actually an algebra isomorphism.

In the present paper, we study the maps \eqref{eq1} and \eqref{eq2} for arbitrary 
DG Lie algebras. We give a general formula for the Drinfeld trace  
in terms of Chern-Simons forms in a convolution DG algebra canonically attached to 
the pair $ (\mfa, \g) $ (see Theorem~\ref{Drintr}). 
Our construction is inspired by an idea of Beilinson \cite{Be} who suggested 
that Chern-Simons classes of canonical $\g$-torsors on convolution algebras should give 
additive analogues of Borel regulator maps\footnote{In a sense, our Drinfeld trace map 
is Koszul dual to the `additive' regulator map, the construction of which
is outlined in \cite{Be}, Section~A6. We will give a detailed account on 
Beilinson's construction in Appendix~A.}. 
As a consequence, we show that the composite map
%
\begin{equation}
\la{eq3}
\Phi_{\g}(\mfa) \circ \Tr_{\g}(\mfa):\, \HC_{\bullet}^{(r)}(\mfa) \to \H_{\bullet}(\mfa, \g)^G \to \H_{\bullet}(\mfa, \h)^{\mathbb W}
\end{equation}
which we refer to as the {\it reduced trace} $ \Tr_{\h}(\mfa)$, 
depends only on the Cartan data $ (\h, {\mathbb W}, P) $ (provided 
$ \Sym(\g^*)^G $ is identified with $ \Sym(\h^*)^{\mathbb W} $ via
the Chevalley isomorphism). If $ \mfa $ is an abelian Lie algebra (of any dimension), 
the cyclic homology groups
$ \HC_{\bullet}^{(r)}(\mfa) $ can be expressed in terms of (algebraic) differential forms
on the vector space $ \mfa $, and the reduced trace $ \Tr_{\h}(\mfa) $ is given 
by a canonical ${\mathbb W}$-invariant differential operator defined on de Rham algebra of 
$ \Sym(\mfa)  $. We will give an explicit combinatorial formula for this operator, computing thus the higher character maps for all symmetric algebras (see Theorem~\ref{trdiffop} and 
Examples~\ref{exs44}). We should mention that this formula plays a crucial role in \cite{BFPRW}, where it is used to verify the Harish Chandra quasi-isomorphism conjecture for the classical Lie algebras $\, \gl_n,\, \mathfrak{sl}_n,\, \mathfrak{so}_{2n+1}\,$ and $\,\mathfrak{sp}_{2n}\,$ 
in the stable limit $\, n \to \infty $; however, it is given in \cite{BFPRW} without a proof.


We now proceed with a summary of the contents of the paper.
Section~\ref{S2} contains preliminary material on de Rham complexes and cyclic homology of commutative algebras and cocommutative coalgebras. Although most of this material is 
well known, we pay a special attention to identifications between various descriptions of cyclic homology. In particular, Theorem~\ref{conj1} provides an explicit formula for the isomorphism between the (reduced) cyclic homology of a symmetric algebra and its Koszul dual symmetric coalgebra. The proof of this theorem is somewhat technical and fairly long; we give it in Appendix~B.

In Section~\ref{SecCSCE}, after brief recollections on derived representation schemes,  we prove our main result (Theorem~\ref{Drintr}), expressing Drinfeld traces in terms of Chern-Simons
forms. We also deduce two consequences -- Lemma~\ref{dHCm} and 
Lemma~\ref{thruTr1} -- which are important for our trace computations. Lemma~\ref{dHCm} implies 
that the reduced trace map $ \Tr_{\h}(\mfa) $ depends only on $ \h $ and the choice of a
${\mathbb W}$-invariant polynomial on $ \h $, while Lemma~\ref{thruTr1} essentially 
reduces the computation of $ \Tr_{\h}(\mfa) $ to the rank one situation.

In Section~\ref{TSA}, we compute the reduced trace maps for any finite-dimensional
abelian Lie algebra $ \mfa $ (or equivalently, for the 
symmetric algebra $A = {\mathcal U}(\mfa) = \Sym(\mfa)$). The main result of this section (Theorem~\ref{trdiffop}) 
gives an explicit formula for these traces in terms of a canonical differential operator 
acting on differential forms on $A$.

In Section~\ref{S5}, we give another application of Theorem~\ref{Drintr}. 
This relies on several results from the literature. First, by a general formula 
from \cite{BKR}, the reduced trace maps for $A = \Sym(\mfa)$ can be expressed
in terms of symmetrized sums of Taylor components of an $A_\infty$-quasi-isomorphism 
inverting the minimal resolution of $A$. Applying a homological perturbation formula
due to Merkulov \cite{M}, we compute these components and transform 
the sums over $ {\mathbb S}_n $ into sums over (equivalence classes of) binary trees. 
Comparing then the result with Theorem~\ref{Drintr} gives an interesting 
combinatorial identity that expresses the sums of $A_\infty$-components 
over binary trees in terms of explicit Chern-Simons forms (see Corollary~\ref{cstree}).

As mentioned above, our main formula for Drinfeld traces can be viewed as  
dual to Beilinson's formula for `additive regulators'. However, in \cite{Be}, 
Beilinson gives only a brief sketch of his construction with no explicit formulas. 
In \cite{F}, Feigin elaborates on ideas of \cite{Be} focusing on current Lie algebras $ {\mathfrak G}^M $ on smooth manifolds; he proposes an interesting conjecture on 
the structure of cohomology of $ {\mathfrak G}^M $ but still gives no explicit formula for 
Beilinson's map. An explicit formula appears in Teleman's paper \cite{Te}, where it is 
given simply as part of a definition ({\it cf.} \cite[(2.2)]{Te}), with no reference to 
\cite{Be}. We bridge this gap in Appendix~A, where we provide a detailed account of~\cite[Section A.6]{Be} and show, in particular, that Teleman's formula indeed arises
from Beilinson's construction. We also construct another natural map from Lie homology of current Lie algebras to cyclic homology of commutative algebras and compare it to Beilinson's
one. The relation between these two maps ({\it cf.} Theorem~\ref{csandabscomp}) plays a 
key role in the proof of our main results in Section~\ref{SecCSCE}.

Finally, the second appendix -- Appendix~B -- contains a detailed proof of Theorem~\ref{conj1}. 
This result seems to be new and may be of independent interest.

\subsection*{Notation} Throughout this paper, $ k $ denotes a base field of characteristic zero. 
An unadorned tensor product $\, \otimes \,$ stands for the tensor product over $k$.
Unless stated otherwise, all differential graded (DG) objects are equipped with differentials of degree $-1$. For a graded vector space $ W $, $\bSym(W)$ denotes the
graded symmetric algebra and $\bSym^{c}(W)$ the graded symmetric coalgebra of $W$.

\section{The de Rham complex  and cyclic homology}
\la{S2}
In this section, we recall some basic facts on cyclic homology of commutative algebras
needed for the present paper. We also record the dual statements for cocommutative coalgebras.
The only (apparently) new result in this section is Theorem~\ref{conj1}, the proof of which
is given in the appendix.

\subsection{The de Rham algebra of a commutative DG algebra}
\la{de Rham}
Let $ \cDGA_k $ denote the category of commutative unital DG $k$-algebras with differential of degree $-1$. Recall that, 
for $A \in \cDGA_k$, the DG module $\Omega^1_A$ of K\"{a}hler differentials 
of $A$ is defined as the free DG $A$-module generated by the symbols $ da $ (for $ a \in A$), 
modulo the relations
$$ 
d(\delta \, a)\,=\,\delta \,da\,,\,\,\,\,\, d(ab)\,=\,da.b+a.db\,\text{.}$$
Here, $\delta$ denotes the differentials intrinsic to $A$ and $\Omega^1_A$. For $a\,\in\,A$ homogeneous, $da$ has the same homological degree in $\Omega^1_A$ as $a$ has in $A$.

Consider the (homologically) graded algebra $\bSym_A \Omega^1_A[1]$. Let $d$ denote the (unique) degree $1$ derivation on $\bSym_A \Omega^1_A[1]$ satisfying $d(a)=da\,,\,\,\,\, d(da)\,=\,0\,,\,\,\,\, \forall \,\,a \in A\,\text{.}$ Let $\delta$ denote the (unique) degree $-1$ derivation on $\bSym_A \Omega^1_A[1]$ induced by the differential intrinsic to $A$. It is easy to verify that $d$ and $\delta$ are square $0$ and (anti)commute. Hence, $(\bSym_A \Omega^1_A[1], \delta, d)$ is an algebra object in the category of mixed complexes: we refer to it as the {\it mixed algebra} of $A$. Note that $k \hookrightarrow \bSym_A \Omega^1_A[1]$ via the unit map $k \hookrightarrow A$. We call the mixed complex $ (\bSym_A \Omega^1_A[1]/k, \delta, d)$ the {\it mixed de Rham complex} of $A$ and denote it by $\drm(A)$.

On the other hand, $A$ may be viewed as a {\it cohomologically graded} algebra by inverting degrees $ A^i = A_{-i} $. In this case, $d+\delta$ can also be viewed as a degree $+1$ differential on $\bSym_A \Omega^1_A[-1]$. We call the {\it cochain algebra} $(\bSym_A \Omega^1_A[-1], d+\delta)$ the {\it de Rham algebra} of $A$ and denote it by $\dr(A)$.

Note that $\bSym^q_A (\Omega^1_A[1])\,\cong\, \wedge^q_A\,\Omega^1_A [k]$: explicitly, this isomorphism is given by
$$ a_0da_1\ldots da_q \mapsto (-1)^{|a_2|+2|a_3|+\ldots +(q-1)|a_q|} a_0da_1\ldots da_q \,\text{.}$$
We refer to the complex $(\wedge^q_A\,\Omega^1_A, \delta)$ as the {\it complex of de Rham $q$-forms of $A$} and denote it by $\Omega^q_A$. Let $\Omega^q_{\bar{A}}$ denote $\Omega^q_A$ when $q>0$ and $A/k$ when $q=0$.

\subsection{Cyclic homology of commutative DG algebras} \la{cychomcommdga}

Recall that if $(\mathcal M, \delta, d)$ is a mixed complex, its {\it cyclic homology} is the homology of the total complex $\mathrm{CC}(\mathcal M)$ of the double complex defined by
\[ \mathrm{C}_{p,q}(\mathcal M)\,=\,   \begin{cases}
                                                                   \mathcal M_{q-p}\, , \hfill &  p \geq 0\\
                                                                    0   \,, \hfill & p <0 \\
                                                               \end{cases} \]
 The horizontal differential $\mathrm{C}_{p,q}(\mathcal M) \rar \mathrm{C}_{p-1,q}(\mathcal M)$ is $d$ and the vertical differential $\mathrm{C}_{p,q}(\mathcal M) \rar \mathrm{C}_{p,q-1}(\mathcal M)$ is $\delta$.

The following theorem relates the cyclic and de Rham homologies of a commutative DG algebra.
\begin{theorem}[\cite{L}, Theorem 5.4.7] 
\la{t1}
Assume that $A\,\in\,\cDGA_k$ is smooth as a graded algebra. Then, $\rHC_{\bullet}(A)$ is canonically isomorphic to the cyclic homology of the mixed de Rham complex $\drm(A)$.
\ethm
There is a natural direct sum decomposition
$$ 
\mathrm{CC}[\drm(A)] \,=\, \bigoplus_{i \geq 0}\, \mathrm{CC}^{(i)}[\drm(A)]\ ,
$$
where
$\mathrm{CC}^{(i)}[\drm(A)]\,:=\, \oplus_{n=i}^{2i}\, \Omega^{2i-n}_{\bar{A}}[n]\,$
is the total complex of the double complex $\mathrm{C}^{(i)}$, where
\[
\mathrm{C}^{(i)}_{p,q}\,=\, \begin{cases}
                                                    [\Omega^{i-p}_{\bar{A}}]_{q-i} \,, \hfill&  p \geq 0 \\
                                                             0\,, \hfill & p<0
                                                     \end{cases} \]
The horizontal differential $\mathrm{C}^{(i)}_{p,q} \rar \mathrm{C}^{(i)}_{p-1,q}$ is $d$ and the vertical differential $\mathrm{C}^{(i)}_{p,q}\rar \mathrm{C}^{(i)}_{p,q-1}$ is $\delta$.
\begin{proposition}[\cite{L}, Proposition~5.4.9]
\la{hodge}
The isomorphism of Theorem~\ref{t1} is compatible with Hodge decompositon: in other words, it induces a canonical isomorphism
$$\rHC^{(i)}(A)\,\cong\,\H_{\bullet}(\mathrm{CC}^{(i)}[\drm(A)])\,\text{.}$$
\eprop
Let $A\,=\,(\Sym(W),\delta)$ where $W$ is a finite-dimensional graded $k$-vector space. Then, the de Rham algebra of $A$ is acyclic with respect to the de Rham differential. In other words,  $\mathrm{CC}^{(i)}[\drm(A)]$ is quasi-isomorphic to $\Omega^i_{\bar{A}}/d\Omega^{i-1}_{\bar{A}}[i]$, with the quasi-isomorphism being induced by the projection
\begin{equation} \la{pcyciso} \mathrm{p}\,:\,\mathrm{CC}^{(i)}[\drm(A)] \,=\, \begin{diagram} \mathrm{Tot}(\mathrm{C}^{(i)}) & \rOnto &
\mathrm{C}^{(i)}_{0,\bullet} \,=\, \Omega^i_{\bar{A}}[i] & \rOnto & \Omega^i_{\bar{A}}/d\Omega^{i-1}_{\bar{A}}[i] \end{diagram}\,\text{.} \end{equation} 
As a consequence of this and Theorem~\ref{t1}, we obtain
\begin{theorem}[\cite{L}, Theorem 5.4.12]
\la{t2}
For $A\,=\,(\Sym(W),\delta)$, there is a canonical isomorphism
$$
\rHC_n(A)\,\cong\, \bigoplus_{i \geq 0} \, \H_{n-i}[(\Omega^i_{\bar{A}}/d\Omega^{i-1}_{\bar{A}}, \delta)]
\,\text{.}
$$
In other words, $\rHC_{\bullet}(A)$ is canonically isomorphic to $\H_{\bullet}[\drm(A)/d\drm(A),\delta]$.
\ethm

Recall that the reduced cyclic homology of $A$ is the homology of Connes' (reduced) cyclic complex $\overline{\mathrm{C}}^{\lambda}(A)$. Since $A\,=\,(\Sym(W),\delta)$ is commutative, there is a Hodge decomposition
$ \overline{\mathrm{C}}^{\lambda}(A) \,=\, \oplus_{i=1}^{\infty}\, \overline{\mathrm{C}}^{\lambda,(i)}(A)\,$.
The isomorphism in Theorem~\ref{t2} is compatible with Hodge decompositon. To be precise, consider the antisymmetrization map
$$\varepsilon\,:\, \bar{A} \otimes \Sym^i(\bar{A}[1]) \rar \bar{A} \otimes \bar{A}[1]^{\otimes i}\,,\,\,\,\, a_0 \otimes a_1 \wedge \ldots \wedge a_i \mapsto \sum_{\sigma \,\in\,\Sb_i} (-1)^{f(\sigma, a_1,\ldots a_i)}(a_0, a_{\sigma(1)},\ldots,a_{\sigma(i)})\,,$$
where $(-1)^{f(\sigma, a_1,\ldots a_i)}$ is obtained from the Koszul sign rule when permuting elements with degrees $|a_1|+1,\ldots,|a_i|+1$ via $\sigma$. By~\cite[Section 2.3.5]{L},
$\varepsilon$ induces a map of complexes
\begin{equation}
\la{epscyclic}
\varepsilon\,:\, \Omega^i_{\bar{A}}/d\Omega^{i-1}_{\bar{A}}[i] \rar \overline{\mathrm{C}}^{\lambda,(i)}(A)\,,
\end{equation}
which is known to be a quasi-isomorphism when $A=(\Sym(W),\delta)$. Its inverse is given by
the composite map
\begin{equation}
\la{picyciso}
I_{\rm HKR}\,:\
\overline{\mathrm{C}}^{\lambda,(i)}(A) \into
\overline{\mathrm{C}}^{\lambda}(A) \onto
A \otimes \bar{A}[1]^{\otimes i}/\mathrm{Im}(1-\tau)
\onto \Omega^i_{\bar{A}}/d\Omega^{i-1}_{\bar{A}}[i]
\ ,
\end{equation}
where the last arrow is defined by $\, (a_0,\ldots ,a_i) \mapsto \frac{1}{i!}\, a_0da_1 \ldots da_i \,$. The quasi-isomorphism \eqref{picyciso} is induced by the classical Hochschild-Kostant-Rosenberg map on Hochschild chain complex of $A$; we therefore call $I_{\rm HKR}$ the
Hochschild-Kostant-Rosenberg map.

\subsection{The de Rham coalgebra of a cocommutative DG coalgebra} \la{code Rham}

We now formally dualize the constructions of Sections~\ref{de Rham},~\ref{cychomcommdga} replacing algebras by coalgebras. Let $C\,\in\,\cDGC_k$. The notion of comodule $\Omega^1_C$ of K\"{a}hler codifferentials is dual to that of the module of K\"{a}hler differentials (cf.~\cite[Section 4]{Q}). More precisely, $\Omega^1_C$ is the set of symmetric elements in the $C$-bicomodule $C \otimes C/\Delta(C)$. There is a universal (degree $0$) coderivation $\,d:\,\Omega^1_C \rar C\,$, 
$\, \omega \mapsto d\omega \,$.  The tensor coalgebra $\mathrm{T}^c_C(\Omega^1_C[-1])$ has a unique degree $1$ coderivation $d$ such that
$\,d(\omega)\,:=\,d\omega\,$ for all $\, \omega \in \Omega^1_C \,$ and
$\, d(\alpha) = 0 \,$ for all $\,\alpha \in C\,$. It also has a degree $-1$ coderivation $\delta$ induced by the (co)differential on $C$. It is easy to check that $d$ and $\delta$ are square zero and (anti)commute. Further, both $d$ and $\delta$ restrict to (co)differentials on $\bSym^c_C \Omega^1_C[-1]$, the largest cocommutative sub-coalgebra of $\Omega^1_C[-1]$. This makes $(\bSym^c_C \Omega^1_C[-1],\delta,d)$ a coalgebra object in the category of mixed complexes. Via the counit of $C$, the mixed coalgebra $(\bSym^c_C \Omega^1_C[-1],\delta,d)$ acquires a counit.  We call the kernel of this counit the {\it mixed de Rham complex} of $C$ and denote it by $\drm(C)$ and refer to $(\bSym^c_C \Omega^1_C[-1],\delta,d)$ as the {\it mixed coalgebra} of $C$. On the other hand, one can view $d+\delta$ as a degree $-1$ (co)differential on $\bSym_C \Omega^1_C[1]$: we refer to the resulting cocommutative DG coalgebra as the {\it de Rham coalgebra} of $C$ and denote it by $\dr(C)$. The image of $\bSym^c_C \Omega^1_C[-1]\,\cap\,[\Omega^1_C[-1]]^{\otimes_C q}$ in $[\Omega^1_C]^{\otimes_C q}[-q]$ will be denoted by $\Omega^q_C[-q]$. We refer to the complex $(\Omega^q_C,\delta)$ as the {\it complex of de Rham $q$ forms of $C$}.  For $C\,\in\,\cDGC_{k}$, let $\Omega^q_{\bar{C}}\,:= \,\Omega^q_C$ for $q>0$, with $\Omega^0_{\bar{C}}$ being the kernel of the counit from $\Omega^0_C$ to $k$.

Recall that if $(\mathcal M, \delta, d)$ is a mixed complex, its {\it negative cyclic homology} is the homology of the complex $\mathrm{CC}^{-}(\mathcal M)$, where  $\mathrm{CC}^{-}(\mathcal M)$ is the total complex of the double complex $\mathrm{C}(\mathcal M)$, where
$$\mathrm{C}_{p,q}(\mathcal M) \,=\, \begin{cases}
                                                                 \mathcal M_{q-p}\,, \hfill & p \leq 0 \\
                                                                        0\,, \hfill & p >0
                                                              \end{cases} $$
The horizontal differential $\mathrm{C}_{p,q}(\mathcal M)  \rar \mathrm{C}_{p-1,q}(\mathcal M)$ is $d$ and the vertical differential $\mathrm{C}_{p,q}(\mathcal M)  \rar \mathrm{C}_{p,q-1}(\mathcal M)$ is $\delta$. The next theorem is dual to Theorem~\ref{t1}.
\bthm \la{t4}
Let $C\,:=\, (\Sym^c (W),\delta)$ where $W$ is a graded $k$-vector space of finite (total) dimension.  Then there is a canonical isomorphism
$$ \rHC_{\bullet}(C)\,\cong\,\H_{\bullet}(\mathrm{CC}^{-}[\drm(C)]) \,\text{.}$$
\ethm
There is a natural direct sum decomposition
$$ \mathrm{CC}^{-}[\drm(C)] \,=\, \bigoplus_{i \geq 0} \,\mathrm{CC}^{-,(i)}[\drm(C)]\,,$$
where
$\mathrm{CC}^{-,(i)}\drm(C))\,:=\, \oplus_{n=i}^{2i}\, \Omega^{2i-n}_{\bar{C}}[-n]$
is the total complex of the double complex $\mathrm{C}^{(i)}$, where
$$ \mathrm{C}^{(i)}_{p,q}\,=\, \begin{cases}
                                               [\Omega^{p+i}_{\bar{C}}]_{q+i} \,,\hfill & p \geq 0 \\
                                                  0\,,\hfill & p<0
                                                  \end{cases} $$
The horizontal differential $\mathrm{C}^{(i)}_{p,q} \rar \mathrm{C}^{(i)}_{p-1,q}$ is $d$ and the vertical differential $\mathrm{C}^{(i)}_{p,q}\rar \mathrm{C}^{(i)}_{p,q-1}$ is $\delta$. Dual to Proposition~\ref{hodge}, we have
\bprop
The isomorphism in Theorem~\ref{t4} is compatible with Hodge decomposition: in other words, it induces an isomorphism
$$ \rHC_{\bullet}^{(i)}(C) \,\cong\, \H_{\bullet}(\mathrm{CC}^{-,(i)}[\drm(C)]) \,,\quad \forall \,i \geq 0\,\text{.}$$
\eprop

Note that when $C=(\Sym^c(W),\delta)$ where $W$ is a finite-dimensional graded vector space,  $\mathrm{CC}^{-,(i)}(\drm(C))$ is quasi-isomorphic to $\Ker(d\,:\, \Omega^i_{\bar{C}} \rar \Omega^{i-1}_{\bar{C}})[-i]$. This quasi-isomorphism is induced by the natural inclusion
$$ \inc\,:\, \Ker(d\,:\,\Omega^i_C \rar \Omega^{i-1}_C)[-i] \hookrightarrow \mathrm{CC}^{-,(i)}(\drm(C))\,\text{.}$$
We thus obtain the following statement (which is dual to Theorem~\ref{t2}).
\bthm \la{t3}
 There is a canonical isomorphism
$$ \rHC_n(C)\,\cong\, \bigoplus_{q \geq 0}\, \H_{n+q}([\Ker(d\,:\,\Omega^q_C \rar \Omega^{q-1}_C),\delta])\,\text{.}$$
In other words, $\rHC_{\bullet}(C)$ is canonically isomorphic to $\H_{\bullet}[\Ker(d\,:\,\drm(C) \rar \drm(C)),\delta]$.
\ethm
Dually to Theorem~\ref{t2}, the isomorphism in Theorem~\ref{t3} respects Hodge decomposition. Further, $\rHC_{\bullet}(C)$ is the homology of Connes' reduced complex $\overline{\mathrm{C}}^{\lambda}(C)$ of $C$. Dually to~\cite[Lemma 1.2]{Q}, $ \overline{\mathrm{C}}^{\lambda}(C)\,=\, \cb(C)_{\n}[1]\,,$
where $\cb(C)$ is the cobar construction of $C$ and $R_{\n}\,:=\, R/(k+[R,R])$ for any $R\,\in\,\DGA_{k/k}$. Dually to Theorem~\ref{t2}, the isomorphism in Theorem~\ref{t3} is by the map
\begin{equation} \la{epscyccoalg}  \varepsilon\,:\, \overline{\mathrm{C}}^{\lambda}(C) \stackrel{\sim}{\rar} \bigoplus_i \, \Ker(d\,:\,\Omega^i(C) \rar \Omega^{i-1}(C))[-i] \,\end{equation}
obtained by taking (graded) linear dual of~\eqref{epscyclic} applied to $E\,:=\Hom_k(C,k)$. The inverse map is given by the (graded) dual of the Hochschild-Kostant-Rosenberg map. We refer to this last map as the {\it co-HKR} map and denote it by $I^c_{\rm HKR}$.

\subsection{The mixed Hopf algebra of a vector space} 
\la{symmalg}

In this subsection, we let $A\,:=\,\bSym (W)$ and $C\,:=\,\bSym^c(W[1])$ where $W$ is a finite-dimensional graded vector space (with trivial  differential). Note that $C$ is Koszul dual to $A$ in the sense of~\cite{LV}. The following lemma is easy to verify.
\blemma \la{kcodiff}
$\Omega^1_C$ is isomorphic the cofree $C$-comodule cogenetared by $W[1]$, i.e, $\Omega^1_C\,\cong\, W[1]\otimes_k C$. Under this isomorphism, the universal coderivation
$d\,:\,\Omega^1_C \rar C$ becomes the map
$$ C \otimes_k W[1] \rar C\,,\,\,\,\, v \otimes \alpha \mapsto v \cdot \alpha   \,,$$
where $\cdot$ denotes the product in $\bSym(W[1])$.
\elemma
Let $\mathcal H(W)\,:=\, \bSym(W \oplus W[1])$, equipped with the graded Hopf algebra structure of the symmetric algebra of the graded vector space $W \oplus W[1]$.
\bprop \la{pcode Rham}
The de Rham differential $d$ on $\mathcal H(W)$ makes $(\mathcal H(W), 0, d)$ a Hopf algebra object in the category of mixed complexes. As an algebra, $(\mathcal H(W), 0, d)$ is the mixed algebra of $A$. As a coalgebra, $(\mathcal H(W), 0,d)$ is the mixed coalgebra of $C$.
\eprop
\bproof
Indeed, as a graded coalgebra, $\mathcal H(W)\,=\, \bSym^c(W) \otimes \bSym^c(W[1])$. Futher, by Lemma~\ref{kcodiff}, for any $v_1,\ldots,v_q\,\in\, W$ and $\alpha\,\in\,\bSym^c(W[1])$, we have
$$d_C(v_1 \ldots v_q \otimes \alpha)\,=\, \sum_i  (-1)^{|v_i|(|v_{i+1}|+\ldots+|v_q|)} (-1)^{|v_1|+\ldots+\hat{|v_i|}+\ldots |+|v_q|} v_1 \ldots \hat{v}_i \ldots v_q \otimes s(v_i)\alpha\,,$$
where $d_C$ is the de Rham codifferential on the mixed coalgebra of $C$ and where $s:W \rar W[1]$ is the operator increasing homological degree by $1$. From this, one sees that $d_C$  {\it is equal to} the de Rham differential $d$. It follows that the de Rham differential is a differential of degree $1$ on the graded Hopf algebra $\mathcal H(W)$ and that $(\mathcal H(W),0,d)$ viewed as a DG coalgebra is equal to the mixed coalgebra of $C$. This verifies the first and third assertions in the desired proposition. The second assertion is obvious.
\eproof

\begin{remark} \la{rem1} By Proposition~\ref{pcode Rham}, the identity map is an isomorphism of mixed complexes between the mixed algebra of $A$ and the mixed coalgebra of $C$. However, under this isomorphism,  the $p$-forms on $A$ having coefficients of polynomial degree $q$ are identified with $q$-forms on $C$ having coefficients of polynomial degree $p$.
\end{remark}

\vspace{1ex}

Let $d_A$ denote the differential $d$ on the mixed de Rham complex $\drm(A) $ of $A$
and let $d_C$ denote the differential $d$ on $ \drm(C)$. Combining the isomorphisms of
Theorem~\ref{t2} and Theorem~\ref{t3} with the identifications of
Proposition~\ref{pcode Rham}, we have
\bprop \la{piso}
There is a natural isomorphism $\rHC_{\bullet}(A) \,\stackrel{\sim}{\to}\,\rHC_{\bullet+1}(C)$
given by the composite map
$$  \rHC_{n}(A)\,\cong \,
\begin{diagram} [\mathrm{Coker}(d_A)]_n & \rTo^{d_A} & [\Ker(d_A)]_{n+1} \end{diagram} \,=\, \ [\Ker(d_C)]_{n+1} \,\cong \, \rHC_{n+1}(C)\ .
$$
\eprop

On the other hand, for any  associative algebra $A\,\in\,\Alg_k$ and any cofibrant resolution $R \stackrel{\sim}{\rar} A$ in $\DGA_k$, there is a quasi-isomorphism of 
complexes\footnote{This quasi-isomorphism is constructed in \cite[Section 4.3]{BKR}, where it is denoted $\,s\tau^{\n}(\theta)$.}
\begin{equation}
\la{bkr422}
T \,:\, \overline{\mathrm{C}}^{\lambda}(A) \rar R_{\n}\, \,
\end{equation}
determined by a twisting cochain $ f: \bB(A) \rar R[1] $ whose components $f_n\,:\,A^{\otimes n} \rar R_{n-1}$ are the Taylor components of an $A_{\infty}$-inverse to the quasi-isomorphism $R \stackrel{\sim}{\rar} A$. Explicitly, \eqref{bkr422} is induced on $n$ chains by the map 
({\it cf.} \cite[Theorem~4.2]{BKR})
$$ T_n\,:\,A^{\otimes n+1} \rar R_{\n}\,,\,\,\,\quad \quad \,\,\,(a_0,\ldots,a_n)\, \mapsto \, \sum_{p \in \Z_{n+1}} (-1)^{nk} [f_n(a_p, a_{1+p}, \ldots, a_{n+p})]\,,$$
where $[f]$ is the image of $f \,\in\,R$ in $R_{\n}$.

Now, if $A\,=\,\bSym(W)$ and $C=\bSym^c(W[1])$, we take $\,R\,=\,\cb(C)\,$ to be the cobar construction of $C$, so that $R_{\n}\,=\, \overline{\mathrm{C}}^{\lambda}(C)[-1]$.
In this case, we have the following result which refines Proposition~\ref{piso}.

\bthm 
\la{conj1}
The isomorphism $\rHC_{\bullet}(A) \stackrel{\sim}{\to} \rHC_{\bullet+1}(C) $
induced by $ T $ is given by
$$ \begin{diagram} \rHC_{\bullet}(A) & \rTo^{I_{\rm HKR}}_{\cong} & [\mathrm{Coker}(d_A)]_{\bullet} &\rTo^{d}_{\cong} &  [\Ker(d_C)]_{\bullet+1} & \rTo^{I^c_{\rm HKR}}_{\cong}  &\rHC_{\bullet+1}(C) \end{diagram}\,,$$
where $I_{\rm HKR}$ and $I^c_{\rm HKR}$ are the Hochschild-Kostant-Rosenberg maps
defined in Section~\ref{cychomcommdga} and Section~\ref{code Rham}, respectively.
\ethm

A detailed proof of Theorem~\ref{conj1} will be given in Appendix~\ref{HKRcoHKR}.

\section{Chern-Simons forms and Drinfeld traces} 
\la{SecCSCE}

\subsection{Derived representation schemes of Lie algebras}
\la{DREP}
We review our basic construction of derived representation schemes of Lie algebras 
and associated character maps. For details and proofs, we refer the reader to 
\cite{BFPRW}, Sections 6 and 7.

\subsubsection{Lie representation functor}
Let $ \g $ be a finite-dimensional Lie algebra over $k$. Given an (arbitrary) Lie 
algebra $ \mfa \in \LAlg_k $, we are interested in classifying the representations
of $ \mfa $ in $ \g $. The corresponding moduli scheme $ \Rep_{\g}(\mfa) $ is defined 
by its functor of points:
\begin{equation*}
\la{repg}
\Rep_{\g}(\mfa):\, \cAlg_k \to \Sets\ , \quad
B \mapsto \Hom_{\tt Lie}(\mfa, \, \g(B))
\end{equation*}
that assigns to a commutative $k$-algebra $B$ the set of families of representations of $ \mfa $
in $ \g $ parametrized by the $k$-scheme $ \Spec(B) $. The functor $ \Rep_{\g}(\mfa) $ is represented 
by a commutative algebra $ \mfa_\g $, which has the following canonical presentation ({\it cf.} \cite[Proposition~6.1]{BFPRW}):
\begin{equation}
\la{ag}
\mfa_{\g} \, = \,\frac{\Sym_k(\mfa \otimes \g^*)}{
\langle\!\langle\, (x\otimes \xi^*_{1})\cdot (y \otimes \xi^*_{2}) -
(y \otimes \xi^*_{1})\cdot (x \otimes \xi^*_{2}) - [x, y] \otimes \xi^* \,\rangle\!\rangle}\ ,
\end{equation}
where $ \g^* $ is the vector space dual to $ \g $ and $\,\xi^* \mapsto \xi^*_{1} \wedge \xi^*_{2}\,$ is the linear map $\,\g^* \to \wedge^2 \g^* $ dual to the Lie bracket on $ \g $. 
The universal representation $\,\varrho_{\g}: \mfa \to \g(\mfa_\g) \,$ is given by the 
natural map
$$
\mfa \to \mfa \otimes \g^* \otimes \g \into \Sym_k(\mfa \otimes \g^*) \otimes \g \onto \mfa_{\g} \otimes \g = \g(\mfa_\g)\ ,\quad
x \mapsto \sum_i\, [x \otimes \xi_i^*] \otimes \xi_i\ ,
$$
where $ \{\xi_i \} $ and $\{\xi_i^*\} $ are dual bases in $ \g $ and $ \g^* $.
The algebra $ \mfa_{\g} $ has a canonical augmentation
$\,\varepsilon: \mfa_{\g} \to k \,$ induced by the zero map $\, \mfa \otimes \g^* \to 0 \,$.  
Thus the assignment $\,\mfa \mapsto \mfa_{\g} \,$ defines a functor with values in the 
category of augmented commutative algebras:
\begin{equation}
\la{repcl}
(\,\mbox{--}\,)_{\g}\,:\ \LAlg_k \to \cAlg_{k/k}\ .
\end{equation}
We call \eqref{repcl} the {\it representation functor} in $ \g $. Geometrically,  one can think of $ (\mfa_{\g}, \varepsilon)  $ as a coordinate ring $ k[\Rep_{\g}(\mfa)] $ of the based affine scheme $ \Rep_{\g}(\mfa) $, with the basepoint corresponding to the trivial representation.

Next, observe that, for any
$ \mfa \in \LAlg_k $, the Lie algebra $ \g $ acts naturally on $ \mfa_{\g} $ by
derivations: this action is induced by the coadjoint action $ \ad_{\g}^* $ of $ \g $
and is functorial in $ \mfa $. We write $\,(\,\mbox{--}\,)_{\g}^{\ad \g}\,:\,
\LAlg_k \to \cAlg_{k/k} $ for the subfunctor of 
$ (\,\mbox{--}\,)_{\g} $ defined by taking the $\ad_{\g}^*$-invariants:
$$
\mfa_{\g}^{\ad \g} := \{x \in \mfa_{\g}\,:\,\ad_{\xi}^*(x) = 0\, ,\, \forall\,\xi\in \g\}\ .
$$
If $ \g $ is a reductive Lie algebra and $ G $ is the associated algebraic group, 
the algebra $ \mfa_{\g}^{\ad \g} $ represents the affine 
quotient scheme $\,\Rep_{\g}(\mfa)/\!/G \,$ parametrizing the closed orbits of 
$ G $ in $\,\Rep_{\g}(\mfa) \,$. 

\subsubsection{Derived functors}
\la{DF}
In general, the scheme $ \Rep_{\g}(\mfa) $ is quite singular.
One way to `resolve singularities' is to replace $ \Rep_{\g}(\mfa) $ by a smooth 
DG scheme $ \DRep_{\g}(\mfa) $ having $ \Rep_{\g}(\mfa) $ as its `0-th homology'. 
This approach to resolution of singularities is advocated in \cite{CK}, where it is 
applied to a number of classical moduli problems in algebraic 
geometry. Our construction of $ \DRep_{\g}(\mfa) $ is more algebraic and inspired 
by \cite{BKR}. The idea is to work with the representation functor \eqref{repcl} 
(instead of the representation scheme $ \Rep_{\g}(\mfa) $) and define $ \DRep_{\g}(\mfa) $ 
in terms of the derived functor of \eqref{repcl} using the `abstract' 
homotopy theory in the category of DG Lie algebras.

In more detail, the functor \eqref{repcl} can be extended to the category of
DG Lie algebras by formula \eqref{ag}:
\begin{equation}
\la{drepg}
(\,\mbox{--}\,)_{\g}\,:\ \DGL_k \to \cDGA_{k/k}\ ,\quad \mfa \mapsto \mfa_{\g}\ .
\end{equation}
It is shown in \cite{BFPRW} that, for a fixed
$ \mfa \in \DGL_k $, the corresponding commutative DG algebra $ \mfa_{\g} $ represents 
an affine DG scheme parametrizing the DG Lie representations of $ \mfa $ in $ \g $. 
The homotopy theories in $ \DGL_k $ and $ \cDGA_{k/k} $ are determined by
the classes of weak equivalences which in both cases are taken to be the quasi-isomorphisms. 
The corresponding homotopy categories $ \Ho(\DGL_k) $ and 
$ \Ho(\cDGA_{k/k}) $ are thus defined by formally inverting all morphisms in 
$ \DGL_k $ and $ \cDGA_{k/k} $ inducing isomorphisms on homology. Now, the key point is that, although the functor \eqref{drepg} is not homotopy invariant 
(it does not preserve quasi-isomorphisms and hence does not descend to $ \Ho(\DGL_k) $), 
it is a left Quillen functor and hence has a well-behaved left derived functor (see \cite[Theorem~6.4]{BFPRW})
$$
\L(\,\mbox{--}\,)_{\g}\,:\,\Ho(\DGL_k) \rar \Ho(\cDGA_{k/k})\ .
$$
The derived functor $ \L(\,\mbox{--}\,)_{\g} $ provides, in a precise sense, the `best possible' 
approximation to $ (\,\mbox{--}\,)_{\g} $ at the level of homotopy categories. 
For a fixed DG Lie algebra $ \mfa \in \DGL_k $, we then define $\,
\DRep_{\g}(\mfa)\,:=\,\L(\mfa)_{\g}\,$. 

In a similar fashion, we construct a derived functor of the invariant functor 
$\,(\,\mbox{--}\,)_{\g}^{\ad \g} $:
$$ 
\L(\,\mbox{--}\,)_{\g}^{\ad \g}\,:\,\Ho(\DGL_k) \rar \Ho(\cDGA_{k/k})\ ,
$$
and define $\,\DRep_{\g}(\mfa)^{\ad\,\g}\,:=\, \L(\mfa)_{\g}^{\ad\,\g}\,$. 
Notice that we use the notation  $ \DRep_{\g}(\mfa) $ and $ \DRep_{\g}(\mfa)^{\ad\,\g} $ for 
the commutative DG algebras in $ \Ho(\cDGA_{k/k}) $ representing the eponymous derived schemes (rather than for the derived schemes themselves). 

As usual, derived functors are computed by applying the original functors to appropriate resolutions. In our case, to compute $ \L(\,\mbox{--}\,)_{\g} $ and $ 
\L(\,\mbox{--}\,)_{\g}^{\ad\,\g} $ we use the cofibrant resolutions in $\DGL_k$, which,
in practice, are given by semi-free DG Lie algebras ({\it cf.} Section~\ref{rhlh}). 
Thus,  we have isomorphisms
$$
\DRep_{\g}(\mfa) \cong (Q\mfa)_{\g}\ ,\quad \DRep_{\g}(\mfa)^{\ad \g} 
\cong (Q \mfa)^{\ad \g} \ ,
$$
where $ Q\mathfrak{a} $ is a(ny) cofibrant resolution of $\mathfrak{a}$ in $\DGL_k$.

Finally, for any DG Lie algebra $ \mfa $, we define the {\it representation 
homology} of $ \mfa $ in $ \g $ by
$$ 
\H_{\bullet}(\mathfrak{a},\g)\,:=\, \H_{\bullet}[\DRep_{\g}(\mathfrak{a})] \ .
$$
This is a graded commutative $k$-algebra, which
depends on $ \mfa $ and $ \g $ (but not on the choice of resolution of $\mfa$).
If $ \mfa \in \LAlg_k $ is an ordinary Lie algebra, then there is an isomorphism of commutative algebras $ \H_0(\mathfrak{a},\g) \cong \mfa_{\g} $ which justifies the name
`derived representation scheme' for $\DRep_{\g}(\mfa)$. In addition, if $ \g $
is a reductive Lie algebra, then
$$
\H_{\bullet}[\DRep_{\g}(\mfa)^{\ad \g}]\, \cong \, \H_{\bullet}(\mathfrak{a},\g)^{\ad \g} .
$$
Thus, the homology of $ \DRep_{\g}(\mfa)^{\ad \g} $ can be viewed as  
the invariant part of the representation homology of $ \mfa $.

\subsubsection{Representation homology vs Lie (co)homology}
\la{rhlh}
Next, we explain how representation homology is related to Lie cohomology.
A key to this relation is a fundamental theorem of Quillen \cite{Q2} that 
establishes a duality (an example of Koszul duality) between
the category of DG Lie algebras and the category $ \cDGC_{k/k} $ of 
cocommutative co-augmented DG coalgebras. Quillen's duality
is given by a pair of adjoint functors
\begin{equation*}
 \la{cbbccom}
\cb_{\mathtt{Comm}}\,:\,
\mathtt{DGCC}_{k/k} \rightleftarrows \DGL_k \,:\,\bB_{\mathtt{Lie}} 
\end{equation*}
inducing an equivalence of the corresponding homotopy categories
$\Ho(\cDGC_{k/k}) $ and $\Ho(\DGL_k) $. The functor $ \bB_{\mathtt{Lie}} $ is defined by the classical Chevalley-Eilenberg complex of a DG Lie algebra 
$\,\bB_{\mathtt{Lie}}(\mathfrak{a})\,:=\, \C(\mathfrak{a};k)\,$: it is a 
Lie analogue of the bar construction for associative algebras.  
The functor $ \cb_{\mathtt{Comm}} $ is defined by taking the
free Lie algebra of the graded vector space $ \bar{C}[-1] $, where 
$ \bar{C} $ is the cokernel of the co-augmentation map 
of a coalgebra $ C \in \cDGC_{k/k} $: this
is a Lie analogue of the classical cobar construction of Adams. 

A DG coalgebra $ C \in \cDGC_{k/k} $ is said to be {\it Koszul dual} to a Lie
algebra $ \mathfrak{a} $ if there is a quasi-isomorphism $\, 
\cb_{\mathtt{Comm}}(C) \stackrel{\sim}{\to} \mathfrak{a}\,$. Such quasi-isomorphisms 
is a common source of cofibrant resolutions in $\, \DGL_k $. In fact, we can choose a cofibrant resolution of $ \mfa $ 
in a functorial way by taking $ C  = \bB_{\mathtt{Lie}}(\mathfrak{a})\,$, but it is often
convenient to work with `smaller' coalgebras.
The cocommutativity of $C$ ensures that $\g^{\ast}(C)\,:=\,\g^{\ast} \otimes C$ has 
a natural structure of a DG Lie coalgebra in the same way as $ \g(B) = \g \otimes B $ has a natural structure of a DG Lie algebra for a commutative algebra $B$.
We let $ \C^c(\g^{\ast}(\bar{C});k) $ denote 
the Chevalley-Eilenberg complex of the Lie coalgebra $\g^{\ast}(C)$, which is defined by a dual construction  
to the classical Chevalley-Eilenberg complex of the Lie algebra $ \g(B) $. We also define
a {\it relative} Chevalley-Eilenberg complex $ \C^c(\g^{\ast}(C),\g^{\ast}; k) $ for the
pair of Lie coalgebras $ \g^{\ast}(C) \onto \g^{\ast} $, which is dual 
to the classical relative Chevalley-Eilenberg complex for the pair of Lie algebras $ \g \into \g(B) $.
Note that $ \C^c(\g^{\ast}(\bar{C});k) $ and $ \C^c(\g^{\ast}(C),\g^{\ast}; k) $ 
are naturally commutative DG algebras in the same way as the classical 
Chevalley-Eilenberg complexes $ \C(\g(B); k) $ and $ \C(\g(B),\g; k) $ are naturally cocommutative DG coalgebras. 

The next theorem is one of the main observations of \cite{BFPRW} (see {\it op. cit.}, Theorem~6.5). 
\bthm 
\la{t6.5bfprw}
Let $C \,\in\, \cDGC_{k/k}$ be Koszul dual to $\mathfrak{a}\,\in\, \DGL_k$. Then there are  isomorphisms in $\Ho(\cDGA_{k/k})${\rm :}
$$ \DRep_{\g}(\mathfrak{a})\,\cong\,\C^c(\g^{\ast}(\bar{C});k)\ ,\quad  \DRep_{\g}(\mathfrak{a})^{\ad\,\g}\,\cong\,
\C^c(\g^{\ast}(C),\g^{\ast};k) \ .
$$
\ethm
As a consequence of Theorem~\ref{t6.5bfprw}, we have
$$ 
\H_{\bullet}(\mathfrak{a},\g)\,\cong\, \H_{\bullet}(\g^{\ast}(\bar{C});k)\,,\,\,\,\,  \H_{\bullet}(\mathfrak{a},\g)^{\ad\,\g}
\,\cong\, \H_{\bullet}(\g^{\ast}(C),\g^{\ast};k) \,\text{.}$$
These isomorphisms can be interpreted by saying that representation homology is Koszul dual 
to the classical Lie homology of current Lie algebras.

\subsubsection{Drinfeld trace maps and cyclic homology} 
\la{secdrintrace}
Our next goal is to describe certain natural maps with values in representation homology.
These maps can be viewed as (derived) characters of finite-dimensional Lie representations.
From now on, we assume that $ \g $ is a reductive Lie algebra over $k$. We let
$ I(\g)\,:=\,\Sym(\g^{\ast})^{\ad\,\g}\,$ denote the space of invariant polynomials 
on $ \g $, and write $  I^r(\g) \subset  I(\g) $ for the subspace of homogeneous
polynomials of degree $ r $.

For a Lie algebra $ \mfa $ and a fixed integer $ r \ge 1 $, we define 
$\,\lambda^{(r)}(\mfa) := \Sym^r(\mathfrak{a})/[\mathfrak{a}, \Sym^r(\mathfrak{a})]\,$, 
which is the space of coinvariants of the adjoint representation of 
$ \mfa $ in $ \Sym^r(\mfa) $.
Note that the vector space $ \lambda^{(r)}(\mfa) $ comes with a natural map
$\,\mfa \times \mfa \times \ldots \times \mfa \to \lambda^{(r)}(\mfa) \,$, which is
the universal symmetric ad-invariant form on $ \mfa $ of degree $ r $.
The assignment $\,\mfa \mapsto \lambda^{(r)}(\mfa)\,$ defines a functor on the category of 
Lie algebras that naturally extends to the category of DG Lie algebras.
Theorem~7.1 of \cite{BFPRW} implies that $  \lambda^{(r)} $ has a left derived functor
$$  
\L \lambda^{(r)}:\, \Ho(\DGL_k) \rar \Ho(\Com_k)\ ,\quad 
\mathfrak{a} \mapsto \lambda^{(r)}(Q\mathfrak{a})\,,
$$
whose homology we denote by $\,\HC_{\bullet}^{(r)}(\mfa)\,$. 
If $ \mfa $ is an ordinary Lie algebra, 
$\,\HC_{\bullet}^{(1)}(\mfa) \,$ is isomorphic to  
the classical (Chevalley-Eilenberg) homology $ \H_\bullet(\mfa, k) $
of $ \mfa $,  while $\,\HC^{(2)}_{\bullet}(\mathfrak{a}) $ 
is the Lie analogue of cyclic homology introduced by Getzler and Kapranov 
in~\cite{GK}.

In general, for any DG Lie algebra $ \mfa \in \DGL_k $, we constructed 
in \cite{BFPRW} a natural isomorphism
\begin{equation} 
\la{liehodgedecomp}   
\rHC_{\bullet}(\mathcal U \mathfrak{a})\,\cong\, \bigoplus_{r=1}^{\infty}\, \HC_{\bullet}^{(r)}(\mathfrak{a}) \ , 
\end{equation}
which can be viewed as a Koszul dual of Hodge decomposition\footnote{More precisely, 
\eqref{liehodgedecomp} is the Hodge decomposition of cyclic homology
of $ \mathcal U \mathfrak{a} $ viewed as a cocommutative coalgebra. We note 
that, if $ \mfa $ is an abelian Lie algebra, \eqref{liehodgedecomp} does {\it not} 
coincide with the classical Hodge decomposition of the symmetric algebra
$ \mathcal U \mathfrak{a} = \Sym(\mfa ) $ ({\it cf.} \eqref{hodgesym}).}
of the (reduced) cyclic
homology of the universal enveloping algebra of $ \mfa $. This isomorphism
justifies the notation for the homology groups $\,\HC_{\bullet}^{(r)}(\mfa)\,$.

Now, observe that, for any commutative algebra $ B $, there is a natural
symmetric invariant $r$-linear form $\,\mfa(B) \times \mfa(B) 
\times \ldots \times \mfa(B) \to \lambda^{(r)}(\mfa) \otimes B \,$ on the
current Lie algebra $ \mfa(B) $. Hence, by the universal property of $ \lambda^{(r)} $,
we have a canonical map
\begin{equation}
\la{canB}
\lambda^{(r)}[\mathfrak{a}(B)] \rar \lambda^{(r)}(\mathfrak{a}) \otimes B \,\text{.}
\end{equation}
Applying $ \lambda^{(r)} $ to the universal representation $ \varrho_\g: \mfa \to  \g(\mfa_{\g}) $ 
and composing  with \eqref{canB}, we define 
\begin{equation} \la{maplambdad}  
\begin{diagram} 
\lambda^{(r)}(\mathfrak{a}) & \rTo 
& \lambda^{(r)}[\g(\mathfrak{a}_{\g})]
& \rTo & \lambda^{(r)}(\g) \otimes \mathfrak{a}_{\g} \end{diagram}\,\text{.} 
\end{equation}
On the other hand, for the Lie algebra $ \g $, we have a canonical (nondegenerate) pairing
\begin{equation}
\la{pair}
I^r(\g) \times \lambda^{(r)}(\g) \to k 
\end{equation}
induced by the linear pairing between $ \g^* $ and $\g$.
Replacing the Lie algebra $ \mfa $ in \eqref{maplambdad} by its cofibrant resolution
$ \mathcal L \sonto \mfa $ and using \eqref{pair}, we define the morphism of complexes
\begin{equation}
\la{dtrm}
\begin{diagram}  
I^{r}(\g) \otimes \lambda^{(r)}(\mathcal L) & \rTo^{\eqref{maplambdad}} & 
I^{r}(\g) \otimes \lambda^{(r)}(\g) \otimes \mathcal L_{\g} 
& \rTo^{\eqref{pair}} & \mathcal L_{\g} \end{diagram} \ .
\end{equation}
For a fixed polynomial $ P \in I^{r}(\g)$, this morphism induces a map
on homology
$\, 
\Tr_{\g}(\mathfrak{a})\,:\, 
\HC_{\bullet}^{(r)}(\mathfrak{a}) \rar \H_{\bullet}(\mathfrak{a},\g)\,
$,
which we call the {\it Drinfeld trace} associated to $P$. (We warn the reader
that $ \Tr_{\g}(\mathfrak{a}) $ does depend on the choice of $P$ but 
we suppress this in our notation.) It is easy to check that
the image of \eqref{dtrm} is contained in the invariant subalgebra
$ \mathcal L_{\g}^{\ad\,\g} $ of $ {\mathcal L}_{\g}$, hence the Drinfeld trace is actually 
a map
\begin{equation} \la{DTr}
\Tr_{\g}(\mathfrak{a}):\, \HC_{\bullet}^{(r)}(\mathfrak{a}) \rar \H_{\bullet}(\mathfrak{a},\g)^{\ad\,\g} \,\text{.}
\end{equation}
We will give an explicit formula for \eqref{DTr} in Section~\ref{MT} below.

\subsection{Chern-Simons forms}
\la{ChernSimons}

In this section, all DG algebras wil be {\it cohomologically graded}. Let $\mathcal A$ be a commutative DG algebra, and let $\g$ be a finite-dimensional Lie algebra.
Recall that a $\g$-valued {\it connection} on $\mathcal A$ is an element $ \theta \in
\mathcal A^1 \otimes \g$; its {\it curvature} is given by
$ \Omega := d\theta+\frac{1}{2}[\theta, \theta]$ in $\mathcal A^2 \otimes \g$, and it is
easy to verify that $d \Omega\,=\,[\Omega, \theta]$, which is usually called the Bianchi identity.

Now, fix $\, P \in I^{r+1}(\g) $, a homogeneous ad-invariant polynomial of degree $ r + 1 $. 
Given $\, \alpha\,\in\,\mathcal A \otimes \Sym^{r+1}(\g) \,$, regard $ P $ as a linear map 
$ \Sym^{r+1}(\g) \to k $ and define $ P(\alpha) $ by applying to $ \alpha $ the evaluation map
$\, \id_{\mathcal A} \otimes \mathrm{ev}_P\,:\, \mathcal A \otimes \Sym^{r+1}(\g) \rar \mathcal A\,$. Thus, $P(\alpha)$ is an element of $\mathcal A$ having the same cohomological degree as
$ \alpha\,$. A simple calculation, using the Bianchi identity, shows that
$\,dP(\Omega^{r+1}) = 0\,$ for any $ \theta \in \mathcal A^1 \otimes \g\,$.
Thus $ P(\Omega^{r+1}) $ is a cocycle in $ \mathcal A $ of degree $ 2r+2 $. In fact, this
cocycle is always exact, and among all coboundaries representing $ P(\Omega^{r+1}) $,
one can specify a natural element $ \TP(\theta) \in {\mathcal A}^{2r+1} $ called
the Chern-Simons form \cite{CS}. Explicitly, this form is defined by the formula (cf. \cite[(3.1)]{CS})
$$
{T\!P}(\theta)\,:=\, (r+1)\,\int_0^1 P(\theta.\Omega_t^{r})dt \ ,
$$
where $ \Omega_t \,=\, t\Omega + \frac{1}{2}(t^2-t)[\theta,\theta] $ is the curvature
for the family of connections $\,\theta_t := t \theta\,$, $\,t \in [0,1]\,$, contracting
$ \theta $ to $ 0 $. We refer to $\TP(\theta)$ as the {\it Chern-Simons form} of $P$ and $\theta$.
A classical calculation (see~\cite[Prop. 3.2]{CS}) gives
\bprop \la{csforms}
$\ d\,\TP(\theta)\,=\, P(\Omega^{r+1})$.
\eprop
We remark that the Chern-Simons form can be also defined directly, without integration,
by the following formula (cf. \cite[(3.5)]{CS})
\begin{equation} \la{csdecomp}
\TP(\theta)\,=\, \sum_{i=0}^{r}\, A_i\,\Psi_{i,P} \ ,
\end{equation}
where $A_i\,:=\,\frac{(-1)^i (r+1)! r!}{2^i (r-i)! (r+1+i)!}$ and $\Psi_{i,P}\,:=\, P(\theta[\theta,\theta]^i\Omega^{r-i})$.

An example of a commutative DG algebra with $\g$-valued connection is the Weil algebra 
$ {\mathcal W}(\g)$ of $\g$. Recall ({\it cf.}~\cite[Section~6.9]{Mein}) that $ 
{\mathcal W}(\g)\,:=\,\bSym(\g^{\ast}[1] \oplus \g^*[2]) \cong 
\wedge(\g^*) \otimes \bSym(\g^{\ast}) $, where the generators $\g^{\ast}$ of $\wedge(\g^{\ast})$ are in cohomological degree $1$ and the generators $\g^{\ast}$ of $\bSym(\g^{\ast})$ are in cohomological degree $2$. The differential is given by the identity map on the generators of cohomological degree $1$ and vanishes on the generators of cohomological degree $2$. It is therefore, easy to see that $ {\mathcal W}(\g)$ is acyclic, i.e, quasi-isomorphic to $k$. The identity map $\g^{\ast} \rar \g^{\ast}\,=\,W^1(\g)$ gives a $\g$-valued connection $\theta_{{\mathcal W}}$ on $ {\mathcal W}(\g)$. There is a second isomorphism $ {\mathcal W}(\g)\,\cong\, \wedge(\g^{\ast}) \otimes \bSym(\g^{\ast})$ such that $\theta_{\mathcal W}$ is still the identity from $\g^{\ast}$ to the $\g^{\ast}$ viewed as the space of generators of $\g^{\ast}$ and the curvature $\Omega_{\mathcal W}$ of $\theta_{\mathcal W}$ is the identity map from $\g^{\ast}$ to the space of generators of $\bSym(\g^{\ast})$. Thus, under this second isomorphism, the element $P(\Omega_{\mathcal W}^{r+1})$ is identified with the degree $2r+2$ element $P\,\in\, \bSym(\g^{\ast}) \subset \wedge(\g^{\ast}) \otimes \bSym(\g^{\ast})$ for any $P\,\in\,I^{r+1}(\g)$. It follows from Proposition~\ref{csforms} that
$P\,=\,d\TP(\theta_{\mathcal W})$.

The Weil algebra is the {\it universal} commutative DG algebra with a $\g$-valued connection. Indeed, a connection $\theta$ on a commutative DG algebra $\mathcal A$ can be viewed as a map f vector spaces $\theta\,:\, \g^{\ast} \rar \mathcal A^1$. This defines a (characteristic) homomorphism $c\,:\,{\mathcal W}(\g) \rar \mathcal A$ by $c(\mu)\,=\,\theta(\mu),c(\hat{\mu})=d_{\mathcal A}\theta(\mu)$ for all $\mu \,\in\g^{\ast}$. Here each $\mu\,\in\,\g^{\ast}$ is viewed as a degree $1$ element of $\wedge(\g^{\ast})$ and $\hat{\mu}\,:=\,d_{\mathcal W}\mu$, i.e,  $\mu$ viewed as a degree $2$ element of $\bSym(\g^{\ast})$. Clearly, $\theta\,=\,c(\theta_{\mathcal W}), \Omega\,=\,c(\Omega_{\mathcal W})$, etc. From this, it follows that for any $P\,\in\,I^{r+1}(\g)$,
$$ P(\Omega^{r+1})\,=\,c(P)\,,\,\,\,\, \TP(\theta)\,=\, c(\TP(\theta_{\mathcal W}))\,\text{.} $$
Thus, $\TP(\theta_{\mathcal W})$ is the universal Chern-Simons form.

\subsection{Main theorem}
\la{MT}
Let $C\,:=\,(\Sym^c(W), \delta)$ be a semi-cofree, cocommutative, conilpotent coaugmented DG coalgebra cogenerated by a finite-dimensional graded vector space $W$. 
Assume further that the corestriction of $\delta$ to $W$ vanishes
on $\Sym^{r}(W)$ for $\,r \gg 0\,$.
Motivated by Beilinson's construction (see~Appendix~A), we consider the convolution (DG) 
algebra
$$
\mathcal A\,:=\, \Hom(\dr(C), \CE^c(\g^{\ast}(\bar{C});k))\,\text{.}
$$
Note that taking bigraded linear duals\footnote{Equipping $W$ and $ W^{\ast}$ with 
weight $1$ makes the coalgebras
$\dr(C)$,$\, \CE(\g(\bar{E});k)$ and the algebras
$\CE^c(\g^{\ast}(\bar{C});k)$,$\,\dr(E)$ bigraded. While taking
bigraded duals, we stick to the convention that homological degrees are
inverted while weights are preserved. Thus, $\dr(C)$ is the
bigraded dual of $\dr(E)$, etc. (even though the differentials are
not necessarily weight-preserving).} gives an 
isomorphism of convolution algebras
\begin{equation} 
\la{grdual} 
\mathcal A\,\cong\, \mathcal A_E\,:=\,\Hom(\CE(\g(\bar{E});k), \dr(E))\,,
\end{equation}
where $E\,:=\,(\Sym(W^{\ast}), \delta^{\ast})$. Equip $\mathcal A$ with a {\it cohomological} grading by inverting all homological degrees.

The DG algebra $\mathcal A$ has a natural $\g$-valued connection $\theta\,\in\,\mathcal A^1 \otimes \g$ given by\footnote{The isomorphism~\eqref{grdual} transforms $\theta$ into the connection~\eqref{conn}}
\begin{equation} \la{conncoalg} \theta(c)\,=\, \sum_{\alpha} \xi_{\alpha}^{\ast} (s^{-1}\bar{c}) \otimes \xi_{\alpha} \,,\end{equation}
where $\{\xi_{\alpha}\}$ is a basis of $\g$ and $ \{\xi_{\alpha}^{\ast}\} $ is the dual basis of $\g^{\ast}$ and $\xi_{\alpha}^{\ast}(\bar{c})\,:=\,\xi_{\alpha}^{\ast} \otimes \bar{c}$ ($\bar{c}$ being the image of $c$ in $\bar{C}$).
 Similarly, the curvature $\Omega\,\in\,\mathcal A^2 \otimes \g$ of the connection $\theta$ vanishes on $\Omega^j_C$ for all $j \neq 1$ and satisfies
\begin{equation} \la{curvcoalg} \Omega(\omega) = \sum_{\alpha} \xi^{\ast}_{\alpha}(s^{-1}{d\omega}) \otimes \xi_{\alpha}\,,   \end{equation}
for $\omega \,\in\,\Omega^1_C$. Further, $\mathcal A$ is a commutative DG algebra, making $\mathcal A \otimes \g$ a DG Lie algebra. Hence, $[\theta, \theta]\,\in\,\mathcal A^2 \otimes \g$. Thus, for any  $P\,\in\,I^{r+1}(\g)$, the Chern-Simons form $\TP(\theta)$ associated with $\theta$ arises as an element of $\mathcal A^{2r+1}$.

Recall that the map $s^{2r}$ increasing homological degree by $2r$ gives a map of graded vector spaces from $\mathrm{CC}^{-,(r)}[\drm(C)]$ to $\dr(C)$. Dually, $s^{2r}$ gives a map of graded vector spaces from $\dr(E)$ to $\mathrm{CC}^{(r)}[\drm(E)]$. Let $\theta_E$ denote the image of the connection $\theta$ of $\mathcal A$ under~\eqref{grdual}. It is known (see Theorem~\ref{csliecyclic}) that the (shifted) Chern-Simons form $s^{2r}\TP(\theta_E)$ gives a map of complexes
$$ s^{2r}\TP(\theta_E)\,:\,\CE(\g(\bar{E});k) \rar \mathrm{CC}^{(r)}[\drm(E)][1] \,\text{.}$$
Taking (bigraded) linear duals, we see that the (shifted) Chern-Simons form $\TP(\theta)s^{2r}$ gives a map of complexes
\begin{equation} 
\la{drtr} 
\TP(\theta)s^{2r}\,:\,  \mathrm{CC}^{-,(r)}[\drm(C)][-1] \rar \CE^c(\g^{\ast}(\bar{C});k) \,\text{.}\end{equation}

 By Theorem~\ref{t3}, the inclusion
 $$ \inc\,:\, \Ker(d\,:\,\Omega^r_C \rar \Omega^{r-1}_C)[-r] \hookrightarrow \mathrm{CC}^{-,(r)}[\drm(C)] $$
 is a quasi-isomorphism. Let $\varepsilon\,:\,\overline{\mathrm{C}}^{\lambda,(r)}(C) \rar \Ker(d\,:\,\Omega^r_C \rar \Omega^{r-1}_C)[-r]$ be as in~\eqref{epscyccoalg}. By Theorem~\ref{t6.5bfprw},
\begin{equation} \la{et6.5bfprw} \DRep_{\g}(\mathfrak{a}) \,\cong\,\CE^c(\g^{\ast}(\bar{C});k)\,,\end{equation}
where $\mathfrak{a}$ is the DG Lie algebra Koszul dual to $C$. On the other hand, recall from~\cite[Prop. 7.4]{BFPRW} that
\begin{equation} \la{kosdual1} \rHC_{\bullet+1}(C)\,\cong\, \rHC_{\bullet}(\mathcal U \mathfrak{a}) \,\text{.}\end{equation}
~\cite[Prop. 7.4]{BFPRW} further implies that the isomorphism~\eqref{kosdual1} respects the Hodge decomposition~\eqref{liehodgedecomp}  to give an isomorphism
\begin{equation} \la{kosdualh} \rHC_{\bullet+1}^{(r)}(C) \,\cong\, 
\HC^{(r+1)}_{\bullet}(\mathfrak{a})\,\end{equation}
With these identifications, $\TP(\theta)s^{2r}$ can be interpeted as a map from $\HC_{\bullet}^{(r+1)}(\mfa)$ to $\H_{\bullet}(\mfa,\g)$.
\bthm \la{Drintr}
For any invariant polynomial $P\,\in\,I^{r+1}(\g)$, the following diagram commutes:
$$\begin{diagram}
\HC^{(r+1)}_{\bullet}(\mathfrak{a}) & \rTo_{\cong}^{\eqref{kosdualh}} & \rHC_{\bullet+1}^{(r)}(C) & \rTo^{\inc \circ \varepsilon}_{\cong} & \H_{\bullet+1}(\mathrm{CC}^{-,(r)}[\drm(C)])\\
 & \rdTo_{\Tr_{\g}(\mathfrak{a})}    &  & & \dTo_{\frac{1}{(r+1)!}\TP(\theta)s^{2r}}\\
  & & \H_{\bullet}(\mathfrak{a},\g) & \rTo^{\mathrm{Thm.}~\ref{t6.5bfprw}}_{\cong} & \H_{\bullet}(\g^{\ast}(\bar{C});k)
\end{diagram}$$
Thus, the map $ \Tr_{\g}(\mfa) $ is given by the formula
\begin{equation*} \la{csshifttr}
\Tr_{\g}(\mfa)\,=\, \frac{1}{(r+1)!}\TP(\theta)s^{2r}\,,
\end{equation*}
where $\theta $ is defined in~\eqref{conncoalg}.
\ethm
\bproof
For brevity, we denote the maps induced on homologies by a given map of complexes by the same symbol as that of the original map of complexes. Let $\mathcal L\,:=\, \cb_{\mathtt{Comm}}(C)$. Note that $\bar{C}[-1] \hookrightarrow \mathcal L$ as graded $k$-vector spaces. Recall from Section~\ref{secdrintrace} (see
~\cite[Sec. 7]{BFPRW} for more details) that the Drinfeld trace was constructed as a map of complexes
$$ \Tr_{\g}(\mathcal{L}) \,:\, \lambda^{(r+1)}(\mathcal L) \rar \C^c(\g^{\ast}(\bar{C});k) \,\text{.} $$
Dually, one obtains a map of complexes
\begin{equation} 
\la{t3.2e1}
 \C(\g(\bar{E});k) \rar \Sym^{r+1}(\mathcal L^c(\bar{E})) \cap \bB(E)^{\natural}\ ,
\text{.}
\end{equation}
where $ \bB(E)^{\natural} $ denotes the cocommutator subspace of the coalgebra $ \bB(E) $.
The isomorphism
\begin{equation} \la{ninverse} \mathrm{C}^{\lambda,(r)}(C)[-1] \,\cong\, \lambda^{(r+1)}(\mathcal L) \end{equation}
inducing~\eqref{kosdualh} on homologies is obtained by taking $k$-linear duals on the isomorphism
$$ \Sym^{r+1}(\mathcal L^c(\bar{E})) \cap \bB(E)^{\natural} \,\cong\,  \mathrm{C}^{\lambda,(r)}(E)[1] \,,$$
whose inverse is explicitly given by the map $N$ which acts on $E[1]^{\otimes n}$ by $1 +\tau+\ldots +\tau^{n-1}$ where $\tau$ is the $n$-cycle $(0,1,\ldots,n-1)$. Let $\varphi_P$ denote the composite map
\[
\begin{diagram} \C(\g(\bar{E});k)  & \rTo^{\eqref{t3.2e1}}& \Sym^{r+1}(\mathcal L^c(\bar{E})) \cap \bB(E)^{\natural} & \rTo^{\cong} & \mathrm{C}^{\lambda,(r)}(E)[1] \end{diagram}\,\text{.}
\]
The composite map
$$ \begin{diagram} \mathrm{C}^{\lambda,(r)}(C)[-1] & \rTo^{~\eqref{ninverse}} & \lambda^{(r+1)}(\mathcal L) & \rTo^{\Tr_{\g}(\mathcal{L})} & \C^c(\g^{\ast}(\bar{C});k) \end{diagram}$$
is thus equal to the map obtained by applying $\varphi_P$ to $E$ and taking (bigraded) linear duals.

It is known (see Theorem~\ref{csandabscomp}) that the diagram
\[
\begin{diagram}
 \CE(\g(\bar{E});k) & & \\
   \dTo^{\frac{1}{(r+1)!}s^{2i}\TP(\theta_E)} &  \rdTo^{\varphi_P}& \\
 \mathrm{CC}^{(r)}[\drm(E)][1] & \rTo^{\varepsilon \circ \mathrm{p}} & \overline{\mathrm{C}}^{\lambda,(r)}(E)[1]
 \end{diagram}
\]
commutes on homologies, where $\mathrm{p}$ is as in~\eqref{pcyciso} and $\varepsilon$ is as in~\eqref{epscyclic}.
The desired result follows immediately from the above fact by taking bigraded linear duals.
\eproof
It is easy to verify that (for the dual statement, see Proposition~\ref{pderhamcs})
\bprop \la{pcoderhamcs}
$\ \TP(\theta)s^{2r} \circ \inc\,=\, P(\theta.\Omega^r)s^{2r}$.
\eprop
As a consequence of Theorem~\ref{Drintr} and Proposition~\ref{pcoderhamcs}, we have
\bcor 
\la{drtrform}
For any $ P \in I^{r+1}(\g) $, the Drinfeld trace map
$ \Tr_{\g}(\mathfrak{a}):\,\HC^{(r+1)}_{\bullet}(\mathfrak{a})
\to \H_{\bullet}(\mathfrak{a},\g)^{\ad \g} $ is given by
\begin{equation} 
\la{edrtr}
\Tr_{\g}(\mathfrak{a})\,=\,\frac{1}{(r+1)!}P(\theta.\Omega^r) s^{2r} \ ,
\end{equation}
where $\theta $ is defined in~\eqref{conncoalg}.
\ecor

We now specialize to the case $ \g = \gl_V $  and use formula \eqref{edrtr} 
to express the derived character maps for associative algebras.

\subsection{The case of $ \gl_V $}
\la{unialg}
Recall that, for any $k$-algebra $A$ and finite-dimensional vector space $V$,
there is a derived scheme $ \DRep_V(A) $, analogous to $ \DRep_{\g}(\mfa)$,  
that parametrizes the representations of $A$ in $V$ as an associative algebra 
(see \cite{BKR}). There exist also natural maps
\begin{equation}
\la{trace}
\Tr_V(A) \,:\, \rHC_{\bullet}(A) \rar \H_{\bullet}(A,V) 
\end{equation}
relating the (reduced) cyclic homology of $A$ to its representation homology
$ \H_{\bullet}(A,V) := \H_\bullet[\DRep_V(A)]$. As shown in \cite{BFPRW}, the representation homology $ \H_\bullet(A,V) $ and the character maps \eqref{trace} can be expressed in Koszul dual terms of Lie coalgebras. To be precise, by \cite[Theorem~3.1]{BFPRW}, there is an isomorphism of algebras
\begin{equation}\la{asslie}
\H_{\bullet}(A,V)\,\cong\, \H_{\bullet}(\gl_V^{\ast}(\bar{C}); k)\ ,
\end{equation}
where $\gl_V^{\ast}(C) := \End(V)^{\ast} \otimes C\,$ is a DG Lie coalgebra defined over 
a (co-associative) coalgebra $ C \in \DGC_{k/k} $ Koszul dual to the algebra $A$. On the other hand, since the cobar construction $ R\,:=\,\cb(C) $
provides a cofibrant resolution of $A$ in $ \DGA_{k/k}$, by a theorem of Feigin-Tsygan,
we can identify
$ \rHC(A) \cong \H_{\bullet}[R_{\n}] $, where $\,R_{\n} :=
\bar{R}/[\bar{R},\bar{R}]\,$. With these identifications, the trace map
$ \Tr_V(A) $ becomes
$$
\Tr_V(A) \,:\, \H_{\bullet}[R_{\n}] \to \H_{\bullet}(\gl_V^{\ast}(\bar{C}); k)\ .
$$

Now, let $A\,=\,\mathcal U\mathfrak{a}$ be the universal enveloping algebra of a Lie algebra
$\mfa$.  In this case, the Koszul dual coalgebra $ C $ can be chosen to be cocommutative, with
$ \mathcal{L} := \cb_{\tt Comm}(C) $, giving a cofibrant resolution of the Lie algebra $ \mathfrak{a} $
in $ \DGL_k $. As a result, \eqref{asslie} becomes an isomorphism  ({\it cf.} \cite[Prop. 6.3]{BFPRW}):
$$ 
\H_{\bullet}(A, V) \cong \H_{\bullet}(\mfa,\gl_V) \ .
$$
Identifying $ \gl_V = \End(V) $, we define polynomials $ \Tr \,\in\,
I^{q+1}(\gl_V)$ by $ \Tr(X,\ldots, X) \,:=\,\Tr_V(X^{q+1})\,$,
where $ \Tr_V:\,\End(V) \to k $ is the usual matrix trace on $V$.
With this choice of invariant polynomials, the direct sum of the Drinfeld traces
gives a map  $\,\oplus_{q \ge 1}\,\HC_{\bullet}^{(q)}(\mathfrak{a}) \to
\H_{\bullet}(\mfa,\gl_V) $, which upon isomorphism \eqref{liehodgedecomp}, coincides
with the trace map \eqref{trace}. By Corollary~\ref{drtrform}, we now conclude
\begin{equation}
\la{traceunivenvalg}
\Tr_V(A)\,= \,
\sum_{q=0}^{\infty} \frac{1}{(q+1)!} \Tr(\theta.\Omega^q)\,s^{2q} \,\text{.}
\end{equation}

\subsection{Reduced trace maps}
\la{DHC}
Unfortunately, despite its simple appearance, the trace formula \eqref{edrtr} is still 
inaccessible to computations.  The problem is that we need to know
an explicit presentation for representation homology which is available
only in a few nontrivial cases. One way to get around this problem is to restrict the Drinfeld 
traces  to `diagonal representations' using
with an appropriate (derived) version of the classical Harish-Chandra homomorphism. 
The derived Harish-Chandra homomorphism plays a crucial role in \cite{BFPRW}
and we briefly recall its construction here. 

We keep the assumption that $ \g $ is a finite-dimensional 
reductive Lie algebra. Let $ \h $ be a Cartan subalgebra of $ \g $, 
and let ${\mathbb W} $ be the corresponding Weyl group. For any Lie algebra $ \mfa $, 
the representation scheme $ \Rep_{\h}(\mfa) $ 
is naturally a (closed) subscheme of $ \Rep_{\g}(\mfa) $. In fact, 
the inclusion $ \h \into \g $ induces a morphism of schemes 
$\, \Rep_{\h}(\mfa)/{\mathbb W} \to \Rep_{\g}(\mfa)/\!/G \,$,
where $\,\Rep_{\h}(\mfa)/{\mathbb W} \,$ is the quotient of $ \Rep_{\h}(\mfa) $ 
relative to the natural action of ${\mathbb W} $ on $ \h $. This yields a
morphism of functors
$\,(\,\mbox{--}\,)^{\ad \g}_{\g} \to  (\,\mbox{--}\,)^{\mathbb W}_{\h}\,$ 
which extends to a morphism of the 
derived functors: $\,\L (\,\mbox{--}\,)^{\ad \g}_{\g} \to  
\L(\,\mbox{--}\,)^{\mathbb W}_{\h}\,$. As a result, for any DG Lie algebra $ \mfa $,
we get a canonical map of commutative DG algebras
\begin{equation}
\la{hch}
\Phi_{\g}(\mfa): \, \DRep_{\g}(\mfa)^{\ad \g}\, \to\, \DRep_{\h}(\mfa)^{\mathbb W} \ .
\end{equation}
which we called the {\it derived Harish-Chandra homomorphism} ({\it cf.} \cite[Section~7]{BFPRW}).
Explicitly, using Theorem~\ref{t6.5bfprw}, we can realize $\,\Phi_{\g}(\mfa)\,$
as a morphism of complexes $\,\C^c(\g^*(C),\,\g^*;\, k) \onto 
\C^c(\h^*(C),\,\h^*;\, k)^{\mathbb W} \,$ corresponding to the projection $\, \g^* \onto \h^* $.
At the homology level, this induces a map 
$\, \H_{\bullet}(\mfa, \g)^{\ad \g} \to  \H_{\bullet}(\mfa, \h)^{\mathbb W} \,$.

Now, let $\, I(\h)^{\mathbb W} := \Sym(\h^*)^{\mathbb W} \,$ denote the space
of $ {\mathbb W}$-invariant polynomials on $ \h $. Recall that 
$ \g^* \onto \h^* $ extends to an isomorphism of algebras 
$\, I(\g) \stackrel{\sim}{\to} I(\h)^{\mathbb W} $ which is called the Chevalley 
isomorphism for $\g$. 
\blemma
\la{dHCm}
For every integer $ r \ge 1 $, the following diagram commutes
$$
\begin{diagram}
 I^{r}(\g) \otimes \HC_{\bullet}^{(r)}(\mathfrak{a}) 
& \rTo^{\,\cong\,} & I^{r}(\h)^{\mathbb W} \otimes \HC_{\bullet}^{(r)}(\mathfrak{a})\\
\dTo^{\Tr_{\g}(\mathfrak{a})}& & \dTo_{\Tr_{\h}(\mathfrak{a})}\\
\H_{\bullet}(\mfa, \g)^{\ad \g} & \rTo^{\Phi_{\g}(\mfa)} & \H_{\bullet}(\mfa, \h)^{\mathbb W}
\end{diagram}
$$
\elemma
\bproof
Let $P\,\in\,I^r(\g)$ and let $P_{\mathbb W}\,\in\,I(\h)^{\mathbb W}$ denote the image of $P$ under the Chevalley isomorphism. Similarly, let $\theta_{\h}$ denote the connection~\eqref{conncoalg} on the convolution algebra $\mathcal A_{\h}\,:=\,\Hom(\dr(C), \CE^c(\h^{\ast}(\bar{C});k))$ and let $\Omega_{\h}$ denote the curvature of $\theta_{\h}$. The inclusion $\h \hookrightarrow \g$ induces an inclusion $\mathcal A_{\h} \otimes \Sym(\h) \hookrightarrow \mathcal A_{\h} \otimes \Sym(\g)$ of commutative DG algebras.  Clearly, for any element $\alpha\,\in\, \mathcal A_{\h} \otimes \Sym(\h)$, $P_{\mathbb W}(\alpha)\,=\,P(\alpha)$, where $\alpha$ on the right hand side is viewed as an element of $A_{\h} \otimes \Sym(\g)$. By~\eqref{edrtr}, 
$$ 
\Tr_{\h}(\mathfrak{a})\,=\,\frac{1}{r!} \, P_{\mathbb W}(\theta_{\h}.\Omega^{r-1}_{\h})\,,\quad \quad\, \Tr_{\g}(\mathfrak{a})\,=\, \frac{1}{r!} \, P(\theta.\Omega^{r-1})\,\text{.}
$$
It therefore suffices to verify that $\Phi_{\g} (\theta.\Omega^{r-1})\,=\,\theta_{\h}.\Omega_{\h}^{r-1}$ as elements of $A_{\h} \otimes \Sym(\g)$. This follows from the fact that $\Phi_{\g}$ is a DG algebra homomorphism and $\Phi_{\g}(\theta)\,=\,\theta_{\h}$, which is easy to verify by direct
calculation. 
\eproof

Lemma~\ref{dHCm} shows that, modulo the Chevalley isomorphism, the composite map
\begin{equation} 
\la{rDTr}
\HC_{\bullet}^{(r)}(\mathfrak{a}) 
\xrightarrow{\Tr_{\g}(\mfa)} \H_{\bullet}(\mathfrak{a},\g)^{\ad\,\g} 
\xrightarrow{\Phi_{\g}(\mfa)}  \H_{\bullet}(\mfa, \h)^{\mathbb W} 
\end{equation}
equals $\, \Tr_{\h}(\mfa) \,$, which depends only on $ \h \,$, $ {\mathbb W} $ and the choice of a 
invariant polynomial $ P \in I(\h)^{\mathbb W} $ but not on the Lie algebra $\g$. 
We call $ \Tr_{\h}(\mfa) $ the {\it reduced trace map}. This map is more accessible 
than the Drinfeld trace, since $ \H_{\bullet}(\mfa, \h)^{\mathbb W}$ is easy to 
compute in many cases. 

In fact, the computation of $ \Tr_{\h}(\mfa) $ reduces to the rank one case.
To be precise, let $ \h = k $ be a one-dimensional Lie algebra with a 
preferred basis. Denote by $\TTr(\mathfrak{a})$ the Drinfeld trace map for $ \mfa $ 
corresponding to the canonical element in $ \Sym(\h^*) $ of degree $r$. Then, 
for an arbitrary $ \h $, the map  $\Tr_{\h}(\mathfrak{a})$ factors through $\TTr(\mathfrak{a})$. 
To see this, choose a Koszul dual coalgebra $C\,\in\,\cDGC_{k/k}$ for $\mfa$, and let $ R\,:=\,\cb(C) $. Then, for a given
$ \h $, choose a linear basis $\{\xi_{\alpha}\} \subset \h $ and define a DG algebra 
homomorphism 
$$
\vartheta_{\h}:\, R_{\ab} \rar \CE^c(\h^{\ast}(\bar{C});k) \otimes \Sym(\h)
$$  
by sending the canonical generators $ s^{-1}c $ of $ R_{\ab} $ to the elements
$$ 
\vartheta_{\h}(s^{-1}c)\,=\, \sum_{\alpha} \xi_{\alpha}^{\ast}(s^{-1}c) \otimes \xi_{\alpha}\ ,
$$
where $ \{\xi_{\alpha}^*\} \subset \h^* $ is the dual basis to $\{\xi_{\alpha}\}$.
Note that the map $ \vartheta_{\h} $ thus defined is independent on the choice of basis $\{\xi_{\alpha}\}$.
Now, for any $P\,\in\,I^r(\h)^{\mathbb W}$, the evaluation at $P$ on the second factor gives a 
map $ \CE^c(\h^{\ast}(\bar{C});k) \otimes \Sym(\h) \rar \CE^c(\h^{\ast}(\bar{C});k)^{\mathbb W}$. 
Write $ P(\vartheta_{\h}) $ for the composition of this map with $\theta_{\h} $:
$$ 
P(\vartheta_{\h})\,:\,R_{\ab} \rar \CE^c(\h^{\ast}(\bar{C});k)^{\mathbb W}\ .
$$
On homology, this induces a map $ \H_{\bullet}(\mfa,k) \rar \H_{\bullet}(\mfa,\h)^W $.  
\blemma 
\la{thruTr1} 
For any $P\,\in\,I(\h)^{\mathbb W}$, the trace map \eqref{rDTr} factors as
$$
\Tr_{\h}(\mfa)\,=\, P(\vartheta_{\h}) \circ \TTr(\mfa)\ .
$$
\elemma
\bproof
Let $\theta_0$ denote $\theta_{\h}$ for $\h:=k$ and let $\Omega_0$ denote the curvature of $\theta_0$. By~\eqref{edrtr}, 
$$ \TTr(\mathfrak{a})\,=\,\frac{1}{r!} [\theta_0.\Omega_0^{r-1}]\,,\quad \quad \,\Tr_{\h}(\mfa)\,=\,\frac{1}{r!}P(\theta_{\h}.\Omega^{r-1}_{\h})\,\text{.}$$
The desired result now follows from the observation that $\theta_{\h}\,=\,\vartheta_{\h}(\theta_0)$. 
\eproof
\begin{example}
Let $\h\,:=\, \h_n $ be the subalgebra of diagonal matrices in $\gl_n$ and let $\mathbb W=\mathbb S_n$ be the symmetric group acting on $\h_n$ by permuting the diagonal entries. Let 
$ P_q\,\in\, I^{q+1}(\h_n)^{{\mathbb S}_n}$ be the symmetric polynomial given by the 
$ (q+1)$-th power sum. Then
$\CE(\h^{\ast}(C);k)^{\mathbb W}\,=\, [R_{\ab}^{\otimes n}]^{\mathbb S_n}$, where $\mathbb S_n$ acts on $R_{\ab}^{\otimes n}$ by permuting the factors. Writing 
$\, \mathbb S^n[R_{\ab}]\,:=\,[R_{\ab}^{\otimes n}]^{\mathbb S_n} $, we see that
$ \sum_{q=0}^{\infty} P_q(\vartheta_{\h})\,:\, R_{\ab} \rar \mathbb S^n[R_{\ab}]$ is
precisely the symmetrization map 
$$ \Sym\,:\, R_{\ab} \rar  \mathbb S^n[R_{\ab}]\,,\quad \quad\, 
r \mapsto \sum_{i=1}^n (1,\ldots, r,\ldots,1) \,,$$
where $r$  in the $i$-th factor is in the $i$-th summand. Now, let $A:=\mathcal U\mfa$. With the above choice of invariant polynomials, the direct sum of the reduced Drinfeld traces $ \Tr_{\h_n}(\mfa)$ gives a map $\oplus_{q \geq 1} \HC_{\bullet}^{(q)}(\mfa) \rar \mathbb S^n[R_{\ab}]$, which upon isomorphism~\eqref{liehodgedecomp}, coincides with the trace map 
$$ \TTr_n(A)\,:\, \rHC_{\bullet}(A) \rar \mathbb S^n[\H_{\bullet}(R_{\ab})] $$
constructed in~\cite[Section 4]{BFPRW}. We therefore, obtain a special case (for universal enveloping algebras) of~\cite[Prop. 4.2]{BFPRW} as a consequence of Lemma~\ref{thruTr1}. 
We remind the reader that, for $n=1$, the reduced trace map 
\begin{equation} \la{redtraceassocalg} \TTr(A)\,:\,\rHC_{\bullet}(A) \rar \H_{\bullet}(R_{\ab}) \end{equation}
coincides with the trace $\Tr_V$ in~\eqref{trace} for $V=k$.
\end{example}

Thus, thanks to Lemma~\ref{thruTr1}, computing the trace map $  \Tr_{\h}(\mfa) $
for any $ \h $ and any invariant polynomial $ P \in I(\h)^{\mathbb W} $
reduces to computing the map $ \TTr(\mfa) = \Tr_{\h}(\mfa) $ for $ \h $ being 
one-dimensional. In the next section, we will given an explicit formula for 
$\TTr(\mfa)$ for an arbitrary abelian Lie algebra $ \mfa $ in terms of
differential forms on $ \Sym(\mfa) $. 

\section{Traces of symmetric algebras}
\la{TSA}
\subsection{Symmetric algebras}
\la{minres}
In this section, $ A := \Sym(W) $ will denote the symmetric
algebra of a vector space $ W $ of finite dimension $ N $. We may think of
$A$ as the universal enveloping algebra of the Lie algebra $ \mfa = W $
with trivial bracket, so that the results of Section~\ref{unialg} will apply.
We will write $ \TTr(\mfa) $ for $ \mfa = W $ as $ \TTr(A) $ or
simply as $ \TTr $ when there is no danger of confusion.

Recall that $A$ has a minimal cofibrant resolution $ R = \cb(C) $ given by the 
cobar construction of the Koszul dual coalgebra $ C = \Sym^c(W[1])$.
The algebra $ R $ is the tensor algebra generated by the vector space
$ \Sym(W[1])[-1] $, whose elements of degree $ k-1 $ we denote by
\begin{equation}
\la{lambda}
\lambda(v_1, v_2, \ldots, v_k) :=  s^{-1}(dv_1 \ldots dv_k)
\,\in\, \Sym^k(W[1])[-1] \ .
\end{equation}

Here, $dv$ denotes $v$ viewed as an element of $W[1]$. With this notation, the differential on $ R $ satisfies
\begin{eqnarray}
&& d \lambda(v_1, v_2) = - [v_1, \,v_2] \ , \la{d1} \\
&& d \lambda(v_1, v_2, v_3) = - [v_1, \lambda(v_2, v_3)] -  [v_2, \lambda(v_3, v_1)] -
[v_3, \lambda(v_1, v_2)]\ . \la{d2}
\end{eqnarray}
In general, one can verify without much difficulty the following formula.
\blemma
\la{diff_res}
The differential $\delta$ on the minimal resolution $R$ of $A=\Sym(W)$ is defined by
\begin{equation*}
\delta\lambda(v_1,\ldots,v_n)=\sum\limits_{\substack{p+q=n\\1\leq p\leq q}}(-1)^p\sum\limits_{\sigma\in\Sh(p,q)}
(-1)^{\sigma}\left[\lambda(v_{\sigma(1)},\ldots,v_{\sigma(p)}),\,\lambda(v_{\sigma(p+1)},\ldots,v_{\sigma(p+q)})\right]\ ,
\end{equation*}
where $\Sh(p,q)$ denotes the set of $(p,q)$-shuffles. Hence, $\delta R \,\subset\,[R\,,\,R]$.
\elemma

Let $ R_{\ab} $ denote the abelianization of $ R $. By Lemma~\ref{diff_res}, $\, R_{\ab} $
is the graded symmetric algebra of $ \Sym(W[1])[-1]$ equipped with zero differential.
Explicitly (omitting the shifts), we can write
\begin{equation}
\la{resab}
R_{\ab} =  \Sym(W) \otimes \bSym(\wedge^2 W \oplus \wedge^3 W \oplus \ldots \oplus \wedge^N W)
\end{equation}
with understanding that the elements of $\wedge^{k}W$ have (homological) degree $ k-1 $.

Note that the de Rham algebra of $A$ can be identified as
$\, \Omega^\bullet_A = \bSym(W \oplus sW) = \Sym(W) \otimes \Lambda(W)\,$,
and for each $ k \ge 1 $, there is a canonical (injective) map
\begin{equation}
\la{drmin}
s^{-1}:\, \Omega^k_A \to R_{\ab}\ ,\quad
a \, dv_1  \ldots dv_k \mapsto a \, \lambda(v_1,\ldots, v_k)\ .
\end{equation}
This map shifts homological degree by $ -1 \,$, whence its notation.

Next, we recall that there is a canonical isomorphism 
({\it cf.} Theorem~\ref{t2})
\begin{equation}
\la{rhco}
\rHC_{\bullet}(A)\,\cong\, \Omega^{\bullet}_A/d\Omega^{\bullet-1}_A\ .
\end{equation}
On the other hand, regarding $A = \Sym(W) $ as the universal enveloping algebra
of the abelian Lie algebra $ \mfa = W $, we have the (dual) 
Hodge decomposition \eqref{liehodgedecomp} of $\rHC_{\bullet}(A)$.
Under the isomorphism \eqref{rhco}, the image of the direct summand $ \HC^{(
r)}_{\bullet}(\mfa) $ in \eqref{liehodgedecomp} is precisely 
$ \Sym^r(W) \otimes \wedge^{\bullet}(W)/d[\Sym^{r+1}(W) \otimes \wedge^{\bullet-1}(W)]$. 
Thus, for the abelian Lie algebra $ \mfa = W$, we have an isomorphism
\begin{equation}
\la{hodgesym}
\HC^{(r)}_{\bullet}(\mfa)\,\cong\, \Sym^r(W) \otimes
\wedge^{\bullet}(W)/d[\Sym^{r+1}(W) \otimes \wedge^{\bullet-1}(W)] \,\text{.}
\end{equation}

Upon the isomorphism~\eqref{liehodgedecomp}, the direct sum of the reduced Drinfeld traces $\TTr(\mfa)\,:\,\HC^{(r)}(\mfa) \rar R_{\ab}$ becomes the reduced trace $\TTr(A)\,:\,\rHC_{\bullet}(A) \rar R_{\ab}$ in~\eqref{redtraceassocalg}. 
Let $\varepsilon$ be as in~\eqref{epscyclic} and let $T$ be the quasi-isomorphism in~\eqref{bkr422}. By Theorem~\ref{conj1}, the isomorphisms $T \circ \varepsilon\,,\, 
\varepsilon^{-1} d\,:\,\Omega^{\bullet}_{\bar{A}}/d\Omega^{\bullet-1}_{\bar{A}} \rar \H_{\bullet}(R_{\n})$ coincide. It follows that $\TTr(A)$ is identified with the composite map
$$\begin{diagram} \Omega^{\bullet}_{\bar{A}}/d\Omega^{\bullet-1}_{\bar{A}} & \rTo^{\varepsilon^{-1} d} &   \H_{\bullet}(R_{\n}) & \rTo^{\TTr(A)} & R_{\ab} \end{diagram}\,, $$

Let $\omega \,\in\,\Omega^p_A$ be a form whose polynomial coefficients are homogeneous of degree $q+1$. The homological degree of $d\omega$  in $\drm(C)$ is $p+1$. On the other hand, $d\omega$ can also be viewed as an element of $\dr(C)$, where its homological degree is $2q+p+1$.  We may therefore, suppress $s^{2q}$ from the notation when we apply~\eqref{traceunivenvalg} to $[d\omega]\,\in\,\H_{p+1}[\drm(C)]$ by reinterpreting $d\omega$ as an element of $\dr(C)$. With these conventions,~\eqref{traceunivenvalg} immediately implies:

\bthm \la{usualtrace}
The reduced trace map $\TTr(A)\,:\,\Omega^{\bullet}_{\bar{A}}/d\Omega^{\bullet-1}_{\bar{A}} \rar R_{\ab}$ is given by
\begin{equation} \la{trace2} \TTr(A)(\omega)\,=\, \sum_{q=0}^{\infty}\,                                                                                        \frac{1}{(q+1)!}\, [\theta.\Omega^q] (d\omega) \,\text{.}\end{equation}
\ethm

We now illustrate Theorem~\ref{usualtrace} for some concrete examples.

\subsection{Traces in low homological degrees} \la{trdiffoplowdeg}

Let $x_1, \ldots, x_N$ be a basis for $W$ and let $dx_1,\ldots,dx_N$ denote the corresponding basis elements in $W[1]$ (i.e, $dx_i\,:=\,sx_i$). In this case, the de Rham coalgebra of $C$ is $\Sym^c(W[1]) \otimes \Sym^c(W[2])$ equipped with the de Rham differential of $A$, with the difference between $\dr(C)$ and $\dr(A)$ being the interpretation of the generators of $A$ as degree $2$ cogenerators of  $\dr(C)$ rather than degree $0$ generators of $A$. Let $f(x_1,\ldots,x_N)$ be a homogenous polynomial of degree $r$ in $x_1,\ldots,x_N$. For notational brevity, the element $f(x_1,\ldots,x_N)dx_{i_1}\ldots dx_{i_p}$ of $\dr(A)$ (which is of cohomological degree $p$ and is in $\Omega^p_A$) shall continue to be denoted by $f(x_1,\ldots,x_N)dx_{i_1}\ldots dx_{i_p}$ when viewed as an element of $\Omega^r_C \subset \dr(C)$ (where its homological degree is $2r+p$). Let $R\,:=\,\cb(C)$. Recall that we denote the element $s^{-1}(dv_1 \ldots dv_p) \,\in\, R_{\ab}$ by $\lambda(v_1,\ldots,v_p)$ for $v_1,\ldots,v_p\,\in\,W$. In what follows, let $\g:=\gl_1\,=\,k$. Choose the element $1$ as the basis as well as the dual basis of $\g$. With these choices,~\eqref{conncoalg} becomes

\begin{equation} \la{conncoalg1}  \theta(f(x_1,\ldots,x_N) dx_{i_1} \ldots dx_{i_p})\,=\, f(0,\ldots,0)\lambda(x_{i_1},\ldots, x_{i_p}) \,\text{.} \end{equation}
Similarly,~\eqref{curvcoalg} becomes
\begin{equation} \la{curvcoalg1} \Omega(f(x_1,\ldots, x_N)dx_{i_1} \ldots dx_{i_p})\,=\, \twopartdef{\lambda(f,x_{i_1}, \ldots ,x_{i_p})}{ f \in W}{0}{f \notin W} \end{equation}

\subsubsection{Homological degree $0$} \la{trdeg0}

 Let $f(x_1,\ldots,x_N)$ be a homogenous polynomial of degree $r+1$ in $A$.
Note that the summands of $df$ are of the form $u_1\ldots u_r du_{r+1}$, where $u_1,\ldots,u_{r+1}\,\in\,W$. By~\eqref{trace2}, we need to evaluate $[\theta.\Omega^q](u_1\ldots u_r du_{r+1})$. Note that
$$\Delta^{q+1}(u_1\ldots u_r du_{r+1})\,=\, \sum \pm u_{S_1}du_{T_1} \otimes \ldots \otimes u_{S_{q+1}}du_{T_{q+1}} \,,$$
where the summation above runs over $S_1 \sqcup \ldots \sqcup S_{q+1}\,=\,\{1,\ldots,r\}$ and $T_1\sqcup \ldots \sqcup T_{q+1}=\{r+1\}$. Hence,
\begin{equation} \la{trace02} [\theta.\Omega^q](u_1\ldots u_r du_{r+1})\,=\, \sum \pm \theta(u_{S_1}du_{T_1}) \Omega(u_{S_2}du_{T_2})\ldots \Omega(u_{S_{q+1}}du_{T_{q+1}}) \,\text{.}\end{equation}
It follows from~\eqref{conncoalg1} and~\eqref{curvcoalg1} that the only summands contributing to the R.H.S of~\eqref{trace02} are those for which $S_1 \,=\,\emptyset$, $|S_2|=\ldots=|S_{q+1}|=1$ and $T_1 \neq \emptyset$. Hence, the R.H.S of~\eqref{trace02} is nonzero only when $q=r$, in which case it equals
$$ r! u_1 \ldots u_{r+1}\,=\, r!\iota_{\epsilon}(u_1\ldots u_r du_{r+1})\,,$$
where $\iota_{\epsilon}$ denotes contraction with the Euler vector field $\epsilon\,:=\, \sum_{i=1}^n x_i\frac{\partial}{\partial x_i}$. Therefore, for $f$ homogenous in $A$ of degree $r+1$, $r \geq 0$,
\begin{equation} \la{trace03} \TTr(A)(f) \,=\, \frac{1}{(r+1)!} [\theta.\Omega^r](df) \,=\, \frac{1}{r+1} \iota_{\epsilon}(df) \,=\, f\,\text{.} \end{equation}

\subsubsection{Homological degree $1$} \la{trdeg1}

Let $\omega= \sum_{i=1}^N f_i dx_i$ where the coefficients $f_i$ are homogenous of degree $r+1$.
Note that the summands of $d\omega$ are of the form $u_1\ldots u_r du_{r+1}du_{r+2}$, where $u_1,\ldots,u_{r+2}\,\in\,W$. By~\eqref{trace2}, we need to evaluate $[\theta.\Omega^q](u_1\ldots u_r du_{r+1}du_{r+2})$. Note that
$$\Delta^{q+1}(u_1\ldots u_r du_{r+1}du_{r+2})\,=\, \sum \pm u_{S_1}du_{T_1} \otimes \ldots \otimes u_{S_{q+1}}du_{T_{q+1}} \,,$$
where the summation above runs over $S_1 \sqcup \ldots \sqcup S_{q+1}\,=\,\{1,\ldots,r\}$ and $T_1\sqcup \ldots \sqcup T_{q+1}=\{r+1,r+2\}$. Hence,
\begin{equation} \la{trace12} [\theta.\Omega^q](u_1\ldots u_r du_{r+1}du_{r+2})\,=\, \sum \pm \theta(u_{S_1}du_{T_1}) \Omega(u_{S_2}du_{T_2})\ldots \Omega(u_{S_{q+1}}du_{T_{q+1}}) \,\text{.}\end{equation}
It follows from~\eqref{conncoalg1} and~\eqref{curvcoalg1} that the only summands contributing to the R.H.S of~\eqref{trace12} are those for which $S_1 \,=\,\emptyset$, $|S_2|=\ldots=|S_{q+1}|=1$ and $T_1 \neq \emptyset$. Hence, the R.H.S of~\eqref{trace12} is nonzero only when $q=r$. In this case, the summands for which $T_1=\{r+1,r+2\}$ contribute
$$r!\lambda(u_{r+1},u_{r+2})u_1 \ldots u_r \,=:\, r! s^{-1}(u_1 \ldots u_r du_{r+1} du_{r+2})$$
to $[\theta.\Omega^r](u_1\ldots u_r du_{r+1}du_{r+2})$. The summands for which $|T_1|=1$ together add up to
$$
r! (\sum_{p=1}^i u_1 \ldots \hat{u}_{p} \ldots u_{r+1}\lambda(u_p, u_{r+2}) - u_1 \ldots \hat{u}_{p}\ldots u_{r}u_{r+2} \lambda(u_p,u_{r+1}) $$ $$\,=\, r!s^{-1}(d\iota_{\epsilon}-2)(u_1 \ldots u_r du_{r+1}du_{r+2}))\,\text{.}$$
Hence,
\begin{eqnarray*}
 \la{trace13} \TTr(A)(\omega)& = & \frac{1}{(r+1)!} [\theta.\Omega^r](d\omega)\\
                          &= & \frac{1}{(r+1)!}[r!s^{-1}d\omega + r!s^{-1}(d\iota_{\epsilon}-2)(d\omega)]\\
                         & = &  \frac{1}{(r+1)!}[r!s^{-1}d\omega + r!s^{-1}(d\iota_{\epsilon}+\iota_{\epsilon} d-2)(d\omega)]\\
                         &=& \frac{1}{(r+1)!}[r!s^{-1}d\omega + r!s^{-1}r(d\omega)]\\
                         &=& s^{-1}d\omega \,\text{.} \end{eqnarray*}

\subsubsection{Homological degree $2$} \la{sectr2}

Let $\omega\,=\,\sum_{i<j} f_{ij}dx_idx_j$ be a two-form whose coefficients $f_{ij}$ are homogenous polynomials of degree $r+1$.
Note that the summands of $d\omega$ are of the form $u_1\ldots u_r du_{r+1}du_{r+2}du_{r+3}$, where $u_1,\ldots,u_{r+3}\,\in\,W$. By~\eqref{trace2}, we need to evaluate $[\theta.\Omega^q](u_1\ldots u_r du_{r+1}du_{r+2}du_{r+3})$. Note that
$$\Delta^{q+1}(u_1\ldots u_r du_{r+1}du_{r+2}du_{r+3})\,=\, \sum \pm u_{S_1}du_{T_1} \otimes \ldots \otimes u_{S_{q+1}}du_{T_{q+1}} \,,$$
where the summation above runs over $S_1 \sqcup \ldots \sqcup S_{q+1}\,=\,\{1,\ldots,r\}$ and $T_1\sqcup \ldots \sqcup T_{q+1}=\{r+1,r+2, r+3\}$. Hence,
\begin{equation} \la{trace22} [\theta.\Omega^q](u_1\ldots u_rdu_{r+1}du_{r+2}du_{r+3})\,=\, \sum \pm \theta(u_{S_1}du_{T_1}) \Omega(u_{S_2}du_{T_2})\ldots \Omega(u_{S_{q+1}}du_{T_{q+1}}) \,\text{.}\end{equation}
It follows from~\eqref{conncoalg1} and~\eqref{curvcoalg1} that the only summands contributing to the R.H.S of~\eqref{trace22} are those for which $S_1 \,=\,\emptyset$, $|S_2|=\ldots=|S_{q+1}|=1$ and $T_1 \neq \emptyset$. Hence, the R.H.S of~\eqref{trace22} is nonzero only when $q=r$. The summands for which $T_1=\{r+1,r+2,r+3\}$ contribute
$$ r!s^{-1}(u_1 \ldots u_r du_{r+1}du_{r+2}du_{r+3})$$
to the R.H.S of~\eqref{trace22}. Similarly, the summands for which $|T_1|=1$ and one of  $|T_2| ,\ldots ,|T_{r+1}|$ is $2$ add up to
$$ r!s^{-1}\sum_{p=1}^ru_1 \ldots \hat{u}_p \ldots  u_r( u_{r+1}\lambda(u_p ,u_{r+2}, u_{r+3})-
u_{r+2}\lambda(u_p ,u_{r+1}, u_{r+3})+ u_{r+3}\lambda(u_p ,u_{r+1},u_{r+2}))$$
$$ =r!s^{-1}(d\iota_{\epsilon}-3)(u_1 \ldots u_r du_{r+1} \ldots du_{r+3}) \,\text{.}$$
Similarly, the summands for which $|T_1|=2$ add up to $-2r! \cdot D^{(2,2)}(u_1 \ldots u_r du_{r+1} \ldots du_{r+3})$ ,
where
$$ D^{(2,2)}(u_1 \ldots u_r du_{r+1} \ldots du_{r+3})\,:=\, $$ $$-\frac{1}{2} \sum_{j=1}^r u_1 \ldots \hat{u}_j \ldots u_r [-\lambda(u_j,u_{r+1})\lambda(u_{r+2},u_{r+3})+ \lambda(u_j,u_{r+2})\lambda(u_{r+1},u_{r+3})-  \lambda(u_j,u_{r+3})\lambda(u_{r+1},u_{r+2})]\,\text{.}$$
Finally the summands for which $|T_1|=1$ and two of $|T_2|,\ldots, |T_{r+1}|$ are $1$ add up to $r!\hat{D}^{(2,2,1)}(u_1.\ldots u_rdu_{r+1} \ldots du_{r+3})$ where
$$ \hat{D}^{(2,2,1)}(u_1.\ldots u_rdu_{r+1} \ldots du_{r+3}) := $$
$$ \sum_{1 \leq k \neq l \leq r} u_1 \ldots \hat{u}_k \ldots \hat{u}_l \ldots u_r [u_{r+3}\lambda(u_k,u_{r+1})\lambda(u_l,u_{r+2}) +
u_{r+2}\lambda(u_k,u_{r+1})\lambda(u_l,u_{r+3})+ u_{r+1}\lambda(u_k,u_{r+2})\lambda(u_l,u_{r+3})]$$
A direct computation shows that
\begin{equation} \la{dopcomp} \hat{D}^{(2,2,1)}\circ d\,=\, -(r-1)D^{(2,2)} \circ d  \end{equation}
on $2$-forms with polynomial coefficients that are homogenous of degree $r+1$. Hence,
\begin{eqnarray*}
 \la{trace23} \TTr(A)(\omega)& = & \frac{1}{(r+1)!} [\theta.\Omega^r](d\omega)\\
                          &= & \frac{1}{(r+1)!}[r!s^{-1}d\omega + r!s^{-1}(d\iota_{\epsilon}-3)(d\omega)- 2r!D^{(2,2)}(d\omega)+r!\hat{D}^{(2,2,1)}(d\omega)]\\
                         & = &  \frac{1}{(r+1)!}[r!s^{-1}d\omega + r!s^{-1}(d\iota_{\epsilon}+\iota_{\epsilon} d-3)(d\omega) -2r!D^{(2,2)}(d\omega)- (r-1)r!D^{(2,2)}(d\omega)]\\
                         &=& \frac{1}{(r+1)!}[r!s^{-1}d\omega + r!s^{-1}r(d\omega)- (r+1)r!D^{(2,2)}(d\omega)]\\
                         &=& s^{-1}d\omega -D^{(2,2)}(d\omega) \,\text{.} \end{eqnarray*}

\subsection{Traces as differential operators}

\subsubsection{}

We aim to provide a formula for reduced traces of $A=\Sym(W)$ in arbitrary homological degree. To this end, for each $ 1 \leq p \leq k \leq \dim(W) $,
let us define a linear map
\begin{equation}
\la{dnko}
W \otimes \Lambda^k(W) \to \Lambda^{p}(W)\otimes \Lambda^{k+1-p}(W)
\end{equation}
by
\begin{equation*}
u \otimes v_1\wedge\ldots\wedge v_k \mapsto
\,
\frac{1}{p!} \sum \limits_{j_1 <\ldots<j_{p-1}}(-1)^{\sum j_s -\frac{(p-1)(p-2)}{2}}\,  (u \wedge v_{j_1}\wedge \ldots \wedge v_{j_{p-1}} ) \,
\otimes \,(v_1 \wedge \ldots \wedge \hat{v}_{j_1}\wedge\ldots \wedge\hat{v}_{j_{p-1}}\wedge \ldots \wedge v_k)\,
\end{equation*}
with convention that this is the identity map if $ p = 1 $.
By duality, \eqref{dnko} gives a canonical map
\begin{equation}
\la{dnk}
\Delta_k^{(p, k+1-p)}:\ \Lambda^k(W)\to W^{\ast} \otimes \Lambda^p(W) \otimes  \Lambda^{k+1-p}(W)\,\text{.}
\end{equation}
Now, for any multi-index $\,(i_1,\ldots, i_m) \in \N^m\,$ such that $i_1+ \ldots+i_m = k+m-1$, we can construct
\begin{equation}
\la{dnki}
\Delta_k^{(i_1,\ldots,i_m)}:\, \Lambda^k(W)\to  \Sym^{m-1}(W^{\ast}) \otimes \Lambda^{i_1}(W)\otimes\ldots\otimes\Lambda^{i_m}(W)
\end{equation}
by iterating \eqref{dnk}:
\begin{equation*}
\Lambda^k(W)\to W^{\ast} \otimes \Lambda^{k+1-i_m}(W)\otimes \Lambda^{i_m}(W)  \to\ldots\to  (W^{\ast})^{\otimes m-1} \otimes \Lambda^{i_1}(W)\otimes\ldots\otimes\Lambda^{i_m}(W)
\end{equation*}
and then projecting $\,(W^*)^{\otimes m-1} \onto \Sym^{m-1}(W^*)\,$.

Finally, interpreting the elements $ \Sym(W^*) $ as constant coefficient differential operators on $ \Sym(W) $, we define
the following differential operator on forms
\begin{equation}\la{diffk}
D_k^{(i_1,\ldots,i_m)}:\,
\Sym(W) \otimes  \Lambda^k W   \xrightarrow{1 \otimes \Delta_k^{(i_1,\ldots,i_m)}}
 \Sym(W) \otimes \Sym^{m-1}(W^{\ast})  \otimes \Lambda^{i_1}(W)\otimes\ldots\otimes\Lambda^{i_m}(W)
\xrightarrow{  {\rm act} \otimes (s^{-1})^{\otimes m} } R_{\ab}
\end{equation}
where $\, {\rm act}:\,
\Sym^{m-1}(W^*) \otimes \Sym(W) \to \Sym(W) \into R_{\ab}\,$ is the action map\footnote{View $\Sym(W) \otimes \wedge^k(W)$ as the module of sections of the bundle $\Omega^k_X$, where $X\,:=\, \mathrm{Spec}[\Sym(W)]$. Similarly view $R_{\ab}$ as the module of sections of the corresponding (homologically graded) vector bundle $\mathcal F$ on $X$. By construction, $D_k^{(i_1,\ldots,i_m)}$ can indeed be viewed as a global section of $\mathcal D_X^{ < m} \otimes \mathbf{Hom}(\Omega^k_X, \mathcal F)$, where $\mathcal D_X^{ < m}$ is the sheaf of differential operators of order less than $m$ on $X$.}
and $ s^{-1} $ is the embedding defined by $ s^{-1}(v_1 \wedge \ldots \wedge v_k) \,:=\, \lambda(v_1,\ldots,v_k)$.

For example, the first order differential operator $\,
D_{k}^{(p,\, k+1-p)}: \Omega^k(A) \to R_{\ab} \,$ is
explicitly given by
$$
(u_1 \ldots u_n)\,
dv_1 \ldots  dv_k \,\mapsto\,
 \frac{1}{p!} \sum_{i=1}^n\,(u_1 \ldots \hat{u}_i \ldots u_n)\sum \limits_{j_1 <\ldots<j_{p-1}}\, \pm
\lambda(u_i, v_{j_1}, \ldots, v_{j_{p-1}}) \lambda(v_1, \ldots, \hat{v}_{j_1}, \ldots,
\hat{v}_{j_{p-1}},\ldots, v_k)\,,
$$

where the sign $\pm $ is given by $ (-1)^{\sum j_s- \frac{(p-1)(p-2)}{2} } $ for $j_1 < \ldots <j_{p-1}$. In particular,
$$
D^{(2,2)}[(u_1 \ldots u_n)dv_1dv_2dv_3] \,=\, \frac{1}{2} \sum_{i=1}^n (u_1 \ldots \hat{u}_i \ldots u_n) [\lambda(u_i,v_1)\lambda(v_2,v_3) -\lambda(u_i,v_2)\lambda(v_1,v_3) +\lambda(u_i,v_3)\lambda(v_1,v_2)]
$$

\subsubsection{The trace formula}

In general, we have
\bthm 
\la{trdiffop}
Let $\omega \,\in\, \Sym^{r+1}(W) \otimes \Lambda^l(W) \,\subset\, \Omega^l_A$ be a $l$-form with homogeneous polynomial coefficients of degree $r+1$. Then,
$$ \TTr(A)(\omega)\,=\, s^{-1}(d\omega)+ \sum_{i_1+\ldots+i_m=l+m} c^{(i_1,\ldots,i_m)}D^{(i_1,\ldots,i_m)}(d\omega)\,,$$
where the sum runs over all tuples $(i_1,\ldots,i_m)$ such that $i_1,\ldots,i_{m-1} \geq 2$ and $i_m \geq 1$ adding up to $l+m$ and where $(r+1)c^{(i_1,\ldots,i_m)}$ depends only on $i_1,\ldots,i_m$.
\ethm
\bproof
By~\eqref{trace2},
$$ \TTr(A)(\omega)\,=\, \sum_{q=0}^{\infty} \frac{1}{(q+1)!} [\theta.\Omega^q] (d\omega) \,\text{.}$$
Note that $d\omega$ is a $k$-linear combination of summands of the form
$u_1\ldots u_{r} du_{r+1} \ldots du_{r+l+1}$, where $u_1,\ldots,u_{r+l+1}\,\in\,W$. Further observe that $$\Delta^{q+1}(u_1\ldots u_r du_{r+1}\ldots du_{r+l+1})\,=\, \sum \pm u_{S_1}du_{T_1} \otimes \ldots \otimes u_{S_{q+1}}du_{T_{q+1}} \,,$$
where the summation above runs over $S_1 \sqcup \ldots \sqcup S_{q+1}\,=\,\{1,\ldots,r\}$ and $T_1\sqcup \ldots \sqcup T_{q+1}=\{r+1,\ldots ,r+l+1\}$. Hence,
\begin{equation} \la{tracegen} [\theta.\Omega^q](u_1\ldots u_r du_{r+1}\ldots du_{r+l+1})\,=\, \sum \pm \theta(u_{S_1}du_{T_1}) \Omega(u_{S_2}du_{T_2})\ldots \Omega(u_{S_{q+1}}du_{T_{q+1}}) \,\text{.}\end{equation}
It follows from~\eqref{conncoalg1} and~\eqref{curvcoalg1} that the only summands contributing to the R.H.S of~\eqref{tracegen} are those for which $S_1 \,=\,\emptyset$, $|S_2|=\ldots=|S_{q+1}|=1$ and $T_1 \neq \emptyset$. Hence, the R.H.S of~\eqref{tracegen} is nonzero only when $q=r$.

Given any tuple $(i_1,\ldots,i_m)$ with $i_1,\ldots,i_{m-1} \geq 2$ and $i_m \geq 1$ such that $i_1+\ldots+i_m = l+m$, the summands on the R.H.S of~\eqref{tracegen} with $|T_1|=i_m$ and $m-1$ among $|T_2|,\ldots,|T_{r+1}|$ being equal to $i_1,\ldots i_{m-1}$ contribute $r! \hat{D}^{(i_1,\ldots,i_m)}(u_1\ldots u_{r} du_{r+1} \ldots du_{r+l+1})$, where
$$  \hat{D}^{(i_1,\ldots,i_m)}(u_1\ldots u_{r} du_{r+1} \ldots du_{r+l+1})\,=\,$$ $$ \sum_{\stackrel{(|T_1|,\ldots,|T_m|)=(i_m,i_1-1,\ldots,i_{m-1}-1)}{T_1 \sqcup \ldots \sqcup T_m=\{ r+1,\ldots,r+l+1\}}}\,\, \sum_{j_1 \neq \ldots \neq j_{m-1}} \pm u_1 \ldots \hat{u}_{j_1} \ldots \hat{u}_{j_{m-1}} \ldots u_r s^{-1}(du_{T_1})s^{-1}(du_{j_1}du_{T_2}) \ldots s^{-1}(du_{j_{m-1}}du_{T_m}) \,\text{.}$$
Note that
$$ \hat{D}^{(i_1,\ldots,i_m)}\,=\, \hat{c}^{(i_1,\ldots,i_m)}D^{(i_1,\ldots,i_m)} \,,$$
where the constant $\hat{c}^{(i_1,\ldots,i_m)}$ depends only on $(i_1,\ldots,i_m)$.\\

The summands on the R.H.S for which $|T_1|=l+1$ contribute $r!s^{-1}(u_1\ldots u_r du_{r+1}\ldots du_{r+l+1})$. Similarly,
$$\hat{D}^{(l,1)}(u_1\ldots u_{r} du_{r+1} \ldots du_{r+l+1})\,=\, \sum_{i=1}^r \sum_{j=1}^{l+1}(-1)^{j-1} u_1 \ldots \hat{u_i} \ldots u_r u_{r+j} s^{-1}(du_idu_{r+1} \ldots \hat{du}_{r+j} \ldots du_{r+l+1}) $$ $$\,=\,s^{-1}(d\iota_{\epsilon}-l-1)(u_1\ldots u_{r} du_{r+1} \ldots du_{r+l+1})\,\text{.}$$

It follows that
$$ \frac{1}{(r+1)!}[\theta.\Omega^r](\eta) = \frac{1}{r+1}s^{-1}(\eta)+ \frac{1}{r+1} s^{-1}(d\iota_{\epsilon}-l-1)(\eta) + \sum_{i_1+\ldots+i_m=l+m} \frac{1}{r+1} \hat{D}^{(i_1,\ldots,i_m)}(\eta) \,$$
$$ = \frac{1}{r+1}s^{-1}(\eta)+ \frac{1}{r+1} s^{-1}(d\iota_{\epsilon}-l-1)(\eta) + \sum_{i_1+\ldots+i_m=l+m} \frac{1}{r+1}\hat{c}^{(i_1,\ldots,i_m)}D^{(i_1,\ldots,i_m)}(\eta) \,,$$
for any $\eta\,\in\, \Sym^r(W) \otimes \Lambda^{l+1}(W) \subset \Omega^{l+1}_A\, $.
Hence, for $\omega \,\in\, \Sym^{r+1}(W) \otimes \Lambda^l(W)$,
 \begin{eqnarray*}
 \frac{1}{(r+1)!}[\theta.\Omega^r](d\omega) &=& \frac{1}{r+1}s^{-1}(d\omega)+ \frac{1}{r+1} s^{-1}(d\iota_{\epsilon}-l-1)(d\omega) + \sum_{i_1+\ldots+i_m=l+m} \frac{1}{r+1}\hat{c}^{(i_1,\ldots,i_m)}D^{(i_1,\ldots,i_m)}(d\omega)\\
  &=& s^{-1}(d\omega) + \sum_{i_1+\ldots+i_m=l+m} \frac{1}{r+1}\hat{c}^{(i_1,\ldots,i_m)}D^{(i_1,\ldots,i_m)}(d\omega)\,\text{.}\end{eqnarray*}
This proves the desired result.
\eproof

\begin{remark}
Comparing the formula in Theorem~\ref{trdiffop} with the computation in Section~\ref{sectr2} when $l=2$,we see that $c^{(2,2)}\,=\,\frac{-2}{r+1}$ and $c^{(2,2,1)}  \neq 0$. However, since $\hat{D}^{(2,2,1)} \circ d\,=\, -(r-1) D^{(2,2)} \circ d$, the summands
$\frac{1}{r+1}\hat{D}^{(2,2)}(d\omega)$ and  $\frac{1}{r+1}\hat{D}^{(2,2,1)}(d\omega)$ of $\TTr(A)(\omega)$ add up to $-D^{(2,2)}(d\omega)$.
\end{remark}

\subsection{Examples}
\la{exs44}
We illustrate the formulas of Section~\ref{trdiffoplowdeg} for polynomial algebras in two and three variables. 

\subsubsection{Polynomials of two variables}
Let $ \dim(W) = 2 $. Choose a basis in $ W $ and identify $ A = k[x,y] \,$. Then $ R = k \langle x,\,y,\,t\rangle $
with $\deg x=\deg y=0$ and $\deg t = 1 $. The differential on $R$
is defined by $\, \delta t=[x,y] \,$, so that $\, t = - \lambda(x,y) \,$,
cf. \eqref{d1}. Section~\ref{trdeg1} says that  $\,\TTr(A):\,\Omega^1_A \to
R_{\ab} \,$ is given by
$$
\TTr(A)[P\,dx + Q\,dy]\, =
\,s^{-1}[(Q_x - P_y) dxdy] =
(Q_x - P_y)\,\lambda(x,y) = (P_y - Q_x)\,t\ .
$$
This formula can also be obtained directly from the explicit formulas of \cite[Ex.~4.1]{BKR}.

\subsubsection{Polynomials of three variables}
Let $ A = k[x,y,z] $. Using the notation of \cite[Ex.~6.3.2]{BFR},
we write the minimal resolution of $A$ in the form $\,R = k \langle x,y,z, \xi,\theta, \lambda, t\rangle\,$, where $\deg x=\deg y=\deg z=0$, $\deg \xi=\deg \theta=\deg\lambda=1$ and $\deg t=2$. The differential on $R$ is defined by
\begin{equation*}
\label{diff_on_R}
\delta\xi=[y,z],\quad \delta\theta=[z,x],\quad \delta\lambda=[x,y];\quad \delta t=[x,\xi]+[y,\theta]+[z,\lambda]
\end{equation*}
Comparing with \eqref{d1} we see that
$$
\xi = \lambda(z,y)\ ,\quad
\theta = \lambda(x,z)\ ,\quad
\lambda = \lambda(y,x)\ , \quad
t = \lambda(x,y,z)\ .
$$
By Section~\ref{trdeg1},  $\,\TTr(A):\,\Omega^1_A \to R_{\ab} \,$ is given by
\begin{eqnarray*}
\TTr(A)[P\,dx + Q\,dy + R\,dz]\,
&=&
\,s^{-1}[(Q_x - P_y) dx dy + (R_y - Q_z) dy  dz +
(P_z - R_x) dz dx] \\*[1.5ex]
&=&
(Q_x - P_y)\,\lambda(x,y) + (R_y - Q_z)\,\lambda(y,z) +
(P_z - R_x)\,\lambda(z,x) \\*[1.5ex]
& = &
(P_y - Q_x)\,\lambda + (Q_z - R_y)\,\xi + (R_x - P_z)\,\theta
\ .
\end{eqnarray*}

Next, to compute $\,\TTr(A)_2\,$ we take $ \omega \in \Omega^2(A) $ in the form
$$
\omega = P\, dxdy + Q \,dy dz + R\, dzdx \ .
$$
The trace formula in Section~\ref{sectr2} implies (after a tedious but straightforward calculation) that
\begin{equation*}
\TTr[\omega]= (P_z+Q_x+R_y)\,t\,+\,(P_{z}+Q_{x}+R_{y})_x\ \theta  \lambda +
(P_{z}+Q_{x}+R_{y})_y\ \lambda  \xi +(P_{z}+Q_{x}+R_{y})_z\ \xi  \theta \ .
\end{equation*}

\section{Reduced traces: a combinatorial description}
\la{S5}
In this section, we will give another formula for reduced trace maps of symmetric
algebras in terms of binary trees. Throughout, we will keep the notation and 
assumptions of the previous section. 

Our starting point is Theorem~4.2 of \cite{BKR} that gives a general formula for 
the derived character maps in terms of Taylor components 
$\,f_{k+1}:\,A^{\otimes (k+1)} \to R $ of an $A_\infty$-quasi-isomorphism 
$f: A \to R $ inverting a DG algebra resolution of $ A $. We apply this 
formula to the minimal resolution $ R := \cb(C) $ of the symmetric algebra 
$ A = \Sym(W) $; as a consequence, we get the following
\bprop
\la{prtr2}
The map $\, \TTr(A) :\, \Omega^\bullet_A/d\Omega^{\bullet-1}_A \to R_{\ab} \,$ is 
given by the formula
\begin{equation}
\la{trfor}
\TTr(A) [a_0 \, da_1  \ldots  da_k] =
\sum_{\sigma \in \Sb_{k+1}}\,
(-1)^{\sigma} \, \bar{f}_{k+1}(a_{\sigma^{-1}(0)}, a_{\sigma^{-1}(1)},\ldots,
a_{\sigma^{-1}(k)})\ ,
\end{equation}
where $\,\bar{f}_{k+1}:\,A^{\otimes (k+1)} \to R \onto R_{\ab} \,$ are
the components of the $A_\infty$-morphism $ A \xrightarrow{f} R \onto R_{\ab} $.
\eprop
\bproof
$\, \TTr(A)\,$ is induced by the composition
$\,\bar{T}:\, \Omega^k_A/d\Omega^{k-1}_A \xrightarrow{\varepsilon_k} \overline{\mathrm{C}}^{\lambda}_k(A) \xrightarrow{\mathrm{can} \circ T} R_{\ab} \,$. Here, $T$ is as in~\eqref{bkr422}.
By~\cite[(4.21)]{BKR}, $\bar{T}$ is given by
$$
\sum_{\tau \in \Z_{k+1}} (-1)^{\tau}
\sum_{\sigma \in \Sb_k}\,(-1)^{\sigma} \, \bar{f}_{k+1} (a_{\tau^{-1}\sigma^{-1}(0)}, a_{\tau^{-1}\sigma^{-1}(1)},\ldots, a_{\tau^{-1}\sigma^{-1}(k)})\ ,
$$
where $ \sigma $ ranges over the subgroup of permutations of $\{0, 1, \ldots, k\}$
preserving $0$ and $ \tau $ ranges over the cyclic subgroup $ \Z_{k+1} $ of $ \Sb_{k+1} $ generated by $ i \mapsto i+1 $. Since  $ \Sb_{k+1} $ is the product of its subgroups
$ \Sb_k $ and $ \Z_{k+1} $, the above sum equals the right-hand side of \eqref{trfor}.
\eproof

\subsection{Merkulov's construction}
\la{SMer}
Let $ \pi \colon R \sonto A $ be a fixed semi-free resolution. Choose a linear section $ f_1 $
of $ \pi $ and identify $A$ with its image in $ R $ under $f_1$. Since $R$ is quasi-isomorphic to $A$, the complex $ R_\bullet $ is acyclic in all degrees $ \geq 1$. For each $ i \ge 0 $, we fix a decomposition of $ R_i $ such that $ R_0 = A\oplus B_0 $ and $ R_i = B_i \oplus L_i $ for $i\geq 1$. Here $B_i=d_{i+1}(R_{i+1})\subseteq R_i$, and $L_i=s_i\left(R_i/B_i\right)\subseteq R_i$, where $s_i\colon R_i/B_i \into R_i$ is a section of the canonical projection $p_i\colon R_i\onto R_i/B_i\,$.

Next, we pick a homotopy $ h \colon R\to R[1] $ between the maps $\id_R$ and $f_1\circ\pi$ satisfying $h_i|_{L_i}=0, h_0|_A=0$ and $h_i|_{B_i}\colon B_i \stackrel{\sim}{\to} L_{i+1}$.
One can construct the components $ h_i: R_i \to R_{i+1} $ of $ h $ inductively:

Since $d_0=0$, the equation for $h_0$ simply is
\begin{equation}
d_1 h_0= \id_R-f_1\pi\ .
\end{equation}
For $n\geq 1$, we have $ \pi|_{R_n}\equiv 0$. Hence  $h_n$ is defined by
\begin{equation}
\la{hn}
d_{n+1} h_n=\id_R - h_{n-1} d_n
\end{equation}
Now, given $ h: R \to R[1] $, for $ i\geq 1 $ we define the operations
$ \mu_i \colon R^{\otimes i}\to R$ by
\begin{itemize}
\item There is no $\,\mu_1\,$, but we formally set $ h \mu_1 := - \id_R $;
\item $\mu_2: R \otimes R \to R $ is the multiplication map $\,
\mu_2(a_1\otimes a_2)=a_1a_2\,$;
\item For $ i \ge 2 $,  $ \mu_i $ is a map of degree $ i - 2 $ defined by
\begin{equation}
\la{merk_mu}
\mu_i := \sum\limits_{\substack{s+t=i\\ s,t\geq 1}} (-1)^{s+1}\mu_2(h\mu_s\otimes h\mu_t)\ .
\end{equation}

\end{itemize}

Finally, for $ k \ge 1 $, we define  $\,f_{k+1}:\,A^{\otimes (k+1)} \to R\,$ by
\begin{equation}
\la{merk}
f_{k+1} := - h_{k-1} \circ \mu_{k+1} \circ f_1^{\otimes(k+1)}\ .
\end{equation}
\noindent

The following observation is due to Merkulov \cite{M} (see also \cite[Prop. 2.3, Lemma~2.5]{LPWZ}).

\bthm
The maps \eqref{merk} define an $ A_{\infty}$-quasi-isomorphism $ f: A \to R $ inverse to $\,\pi\,$.
\ethm

\begin{remark}
In general, if $ R $ is any DG algebra, the above construction also yields
a (minimal) $A_\infty$ structure on  $ \H_\bullet(R) \,$.
The corresponding higher multiplications are
defined by  $ m_k := \pi\circ \mu_k \circ f_1^{\otimes(k+1)}\,$, $\, k\ge 2 $.
In the case when $ \H_\bullet(R) = A$ is an ordinary algebra, the operations $\, m_3, m_4, \ldots $  are trivial for degree reasons, while $ m_2 $ coincides with the induced multiplication
on $ A $, since $ m_2(a_1, a_2)=\pi(f_1(a_1)f_1(a_2)) = \pi(f_1(a_1))\pi(f_1(a_2)) = a_1 a_2 $.
\end{remark}

\subsection{Traces and binary trees}

Substituting \eqref{merk} into formula \eqref{trfor} we  get
\begin{equation}
\label{Tr_k}
\TTr(a_0da_1\ldots da_k)\, = \,\sum\limits_{\sigma\in \Sb_{k+1}}(-1)^{1+\sigma}\,
\bar{h}_{k-1} \, \mu_{k+1}\left(f_1(a_{\sigma(0)}),\dots,f_1(a_{\sigma(k)})\right)
\end{equation}
Merkulov's construction provides us with the recursive formula \eqref{merk_mu} for $\mu_{k+1}$, and hence for $f_{k+1}$, in terms of operations $\mu_i$ with $i<k+1$ and homotopy $h$.

 By $\PBT_k$ we denote the set of rooted planar binary trees with $k+1$ leaves. Operation $f_T$ for a tree $T\in\PBT_k$ is defined in the following way. First of all, we will label all the leaves and internal vertices in the following way. Every leaf we will label by $0$. After that, if a vertex $v$ has left son with label $l$ and right son with label $r$, then we label $v$ by $l+r+1$. After that, we insert $f_1$ into each leaf, and if a vertex $v$ was labeled by some number $l$, we insert $h_{l-1}\mu_2$ into $v$. In the very last vertex (the one that is adjacent with the root) we insert $-h_{k-1}\mu_2$. Then moving along the tree down to the root we can read off the map $f_{k+1}$. For example, $f_T$ for the trees
$$
T_1=\arbreBA\hspace{10mm}T_2=\arbreAB
$$
will be just $f_{T_1}(a_0,a_1,a_2)=-h_1\mu_2(f_1(a_0)\otimes h_0\mu_2(f_1(a_1)\otimes f_1(a_2)))=-h_1(\tilde{a}_0\cdot h_0(\tilde{a}_1\cdot\tilde{a}_2))$ and $f_{T_2}(a_0,a_1,a_2)=-h_1(h_0(\tilde{a}_0\cdot \tilde{a}_1)\cdot\tilde{a}_2)$.

There is an algorithm how to determine the sign $(-1)^T$ that corresponds to a tree $T$. First we label leaves by `$+1$'. After that, for any vertex $v$ that has left son labeled by a sign $l$ and right son labeled by $r$, we label $v$ by $l\cdot r\cdot (-1)^{s+1}$, where $s$ is the number of leaves to the left from $v$. Then the sign of the tree $(-1)^T$ is by definition the sign of the last vertex (the one that is adjacent with the root). For example, for the trees $T_1$ and $T_2$ above we will have $(-1)^{T_1}=1$ and $(-1)^{T_2}=-1$.

The construction defined above (almost) coincides with the construction given in \cite[Sect.~6.4]{ks}. The only difference is that we expanded all higher multiplications $\mu_i$ with $i>2$.

\blemma[{\it cf.} \cite{ks}]
\label{f_k_formula_trees}
\begin{equation}
\la{fk1}
f_{k+1}=\sum\limits_{T\in\PBT_k}(-1)^T f_T
\end{equation}
\elemma
\bproof
The proof can be obtained by easy induction on the number of vertices.
\eproof
Now if we apply the result of the lemma \ref{f_k_formula_trees} to the formula \eqref{Tr_k} we will get the following expression for the reduced trace:
\begin{equation}
\label{Tr_k_trees}
\TTr(a_0da_1 \dots da_k)\,=\, \sum\limits_{\sigma\in \Sb_{k+1}}(-1)^\sigma\sum\limits_{T\in\PBT_k}(-1)^T\, \bar{f}_{T}\left(\tilde{a}_{\sigma(0)},\dots,\tilde{a}_{\sigma(k)}\right).
\end{equation}

Two {\it labeled} planar rooted binary trees $(\sigma, T)$ and $(\sigma',T')$ are {\it equivalent} if there exists a rooted tree isomorphism $\varphi\,:T \rar T'$ that preserves the labels on leaves (labeling is given by a choice of $\sigma \,\in\,\Sb_{k+1}$ which we think of as labels on $k+1$ leaves). Let's denote the set of equivalence classes of pairs $(\sigma, T)$ by $\clLPBT_k$.

For any tree $T$ define $[f]_{T}$ to be a map, obtained from $f_{T}$ by replacing any $\mu_2(a\otimes b)$ appearing in $f_{T}$ by $[a,b]$. For example, if $f_{T}(a_0,a_1,a_2)=-h_1(h_0(\tilde{a}_0\cdot \tilde{a}_1)\cdot\tilde{a}_2)$, then $[f]_T=-h_1[h_0[\tilde{a}_0, \tilde{a}_1],\tilde{a}_2]$.

\blemma \la{treeform}
\begin{equation}
\label{Tr_k_trees_comm}
\TTr(a_0da_1 \dots  da_k)\, = \, \sum\limits_{[\sigma,T]\in \clLPBT_k}(-1)^{\sigma_0}\cdot(-1)^{T_0}\,[\,\bar{f}\,]_{T_0}\left(\tilde{a}_{\sigma_0(0)},\dots,\tilde{a}_{\sigma_0(k)}\right).
\end{equation}
Here $(\sigma_0,T_0)$ is a representative of the class $[\sigma,T]$.
\elemma
\bproof
Straightforward induction on the number of vertices.
\eproof

\vspace{10pt}

As an example, let us consider the case $k=3$. There are $5$ elements in $\PBT_3$:
$$
T_1=\arbreABC\quad T_2=\arbreBAC\quad T_3=\arbreCAB\quad T_4=\arbreCBA\quad T_5=\arbreACA
$$

Their signs are going to be $(-1)^{T_1}=-1$, $(-1)^{T_2}=+1$, $(-1)^{T_3}=-1$, $(-1)^{T_4}=+1$ and $(-1)^{T_5}=-1$.

There are $15$ equivalence classes in $\clLPBT_3$. There are $3$ classes $[(\sigma,T)]$ with $T=T_5$ and $\sigma\in\Sigma_1:=\{(0123),(0213),(0312)\}$. There are $12$ classes $[(\sigma,T)]$ with $T=T_4$ and $\sigma\in\Sigma_2=\{(\sigma(0),\sigma(1),\sigma(2),\sigma(3))\in \Sb_4\mid\sigma(2)<\sigma(3)\}$. So for $\TTr_3$, we will have the following explicit formula
\begin{eqnarray}
\TTr(a_0da_1da_2 da_3)\,=\,\sum\limits_{\sigma\in\Sigma_1}(-1)^{\sigma}\,\bar{h}_2\left[h_0[\tilde{a}_{\sigma(0)}, \tilde{a}_{\sigma(1)}],h_0[ \tilde{a}_{\sigma(2)}, \tilde{a}_{\sigma(3)}]\right]\\
+\sum\limits_{\sigma\in\Sigma_2}(-1)^{1+\sigma}\,\bar{h}_2\left[ \tilde{a}_{\sigma(0)},h_1[\tilde{a}_{\sigma(1)},h_0[\tilde{a}_{\sigma(2)},\tilde{a}_{\sigma(3)}]]\right]\nonumber
\end{eqnarray}

\noindent
More explicitly,
\begin{eqnarray*}
\TTr(a_0da_1 da_2 da_3)&=&\bar{h}_2\left[h_0[\tilde{a}_0,\tilde{a}_1],h_0[\tilde{a}_2,\tilde{a}_3]\right]-\bar{h}_2\left[h_0[\tilde{a}_0,\tilde{a}_2],h_0[\tilde{a}_1,\tilde{a}_3]\right]+ \bar{h}_2\left[h_0[\tilde{a}_0,\tilde{a}_3],h_0[\tilde{a}_1,\tilde{a}_2]\right] \\*[1ex]
&-&\bar{h}_2\left[\tilde{a}_0,h_1\left[\tilde{a}_1,h_0[\tilde{a}_2,\tilde{a}_3]\right]\right]
+\bar{h}_2\left[\tilde{a}_0,h_1\left[\tilde{a}_2,h_0[\tilde{a}_1,\tilde{a}_3]\right]\right]
-\bar{h}_2\left[\tilde{a}_0,h_1\left[\tilde{a}_3,h_0[\tilde{a}_1,\tilde{a}_2]\right]\right]\\*[1ex]
&+&\bar{h}_2\left[\tilde{a}_1,h_1\left[\tilde{a}_0,h_0[\tilde{a}_2,\tilde{a}_3]\right]\right]
-\bar{h}_2\left[\tilde{a}_2,h_1\left[\tilde{a}_0,h_0[\tilde{a}_1,\tilde{a}_3]\right]\right]
+\bar{h}_2\left[\tilde{a}_3,h_1\left[\tilde{a}_0,h_0[\tilde{a}_1,\tilde{a}_2]\right]\right]\\*[1ex]
&-&\bar{h}_2\left[\tilde{a}_1,h_1\left[\tilde{a}_2,h_0[\tilde{a}_0,\tilde{a}_3]\right]\right]
+\bar{h}_2\left[\tilde{a}_1,h_1\left[\tilde{a}_3,h_0[\tilde{a}_0,\tilde{a}_2]\right]\right]
-\bar{h}_2\left[\tilde{a}_2,h_1\left[\tilde{a}_3,h_0[\tilde{a}_0,\tilde{a}_1]\right]\right]\\*[1ex]
&+&\bar{h}_2\left[\tilde{a}_2,h_1\left[\tilde{a}_1,h_0[\tilde{a}_0,\tilde{a}_3]\right]\right]
-\bar{h}_2\left[\tilde{a}_3,h_1\left[\tilde{a}_1,h_0[\tilde{a}_0,\tilde{a}_2]\right]\right]
+\bar{h}_2\left[\tilde{a}_3,h_1\left[\tilde{a}_2,h_0[\tilde{a}_0,\tilde{a}_1]\right]\right]\phantom{.+}
\end{eqnarray*}

By Theorem~\ref{usualtrace} and Lemma~\ref{treeform}, we have
\bcor \la{cstree}
$$  \sum\limits_{[\sigma,T]\in \clLPBT_k}(-1)^{\sigma_0}\cdot(-1)^{T_0}\,[\,\bar{f}\,]_{T_0}\left(\tilde{a}_{\sigma_0(0)},\dots,\tilde{a}_{\sigma_0(k)}\right) \,=\, \sum_{q=0}^{\infty} \frac{1}{(q+1)!} [\theta.\Omega^q](da_0da_1 \ldots da_k)\,\text{.}$$
\ecor

\appendix

\section{Chern-Simons forms, Lie and cyclic homology}

We now give a detailed exposition of the construction of an additive analog of the Borel regulator map as outlined in~\cite[Sec. A.6]{Be}. We then compare this map with a related construction dual to the Drinfeld traces (see Theorem~\ref{csandabscomp} below).

\subsection{The convolution algebra} \la{conv}

Let $A$ be a commutative DG algebra and let $\g$ be a finite-dimensional Lie algebra. Invert degrees to turn $A$ into a {\it cohomologically graded} DG algebra. Then, the Chevalley-Eilenberg complex $\CE(\g(A);k)$ is a cocommutative (cohomologically graded) DG coalgebra.  As a result, one has the commutative DG {\it convolution} algebra
$$ \mathcal A\,:=\, \mathbf{Hom}(\CE(\g(A);k),\dr(A))\,\text{.}$$
Note that as cohomologically graded algebras, $ \mathcal A\,\cong\, \oplus_{i,j} \mathcal A^{i,j} $
where
$$\mathcal A^{i,j}\,:=\,\mathbf{Hom}(\wedge^j \g(A), \Omega^i_A)[-i-j] \,\text{.}$$
Equip $\mathcal A$ with the connection $\theta\,\in\, (\mathcal A^{0,1})^1 \otimes \g \subset \mathcal A^1 \otimes \g$ given by the formula
\begin{equation} \la{conn} \theta(\xi \otimes A)\,=\, a \otimes \xi \,\in\, A \otimes \g\,,\,\,\,\,\,\forall\,\, a \in A\,,\,\xi \in \g\,\text{.} \end{equation}
The following proposition follows from a straightforward computation.
\bprop \la{pcurv}
The curvature
 $\Omega$ of $\theta$ lies in the summand
$\mathcal A^{1,1} \otimes \g$ of $\mathcal A^{2} \otimes \g$. Explicitly,
$$\Omega \in \Hom_{\c}(\g(A)[1], \Omega^1_A[-1])\,,\,\,\,\, \xi \otimes a \mapsto da \otimes \xi \,\text{.}$$
Similarly, the element $[\theta\,,\,\theta]\,\in\, \mathcal A^{0,2} \otimes \g$ is given by
$$  [\theta, \theta] [(\xi_1 \otimes a_1) \wedge (\xi_2 \otimes a_2)]\,=\, -2 (-1)^{|a_1|} a_1a_2 \otimes [\xi_1,\xi_2]\,\text{.}$$
\eprop

\bprop \la{pcurv2}
The element $\Omega^n \in \mathcal A^{n,n} \otimes \g^{\otimes n} = \Hom_{\c}(\wedge^n\g(A)[n], \Omega^n_A[-n]) \otimes \g^{\otimes n}$ is given by $$
(\xi_1 \otimes a_1) \wedge \ldots \wedge(\xi_n \otimes a_n) \mapsto \sum_{\sigma \in \Sb_n}  da_1 \ldots da_n \otimes \xi_{\sigma(1)} \otimes \ldots \otimes \xi_{\sigma(n)} $$
\eprop
\bproof
By Proposition~\ref{pcurv}, the summand of $\Delta^{(n)}[(\xi_1 \otimes a_1) \wedge \ldots \wedge( \xi_n \otimes a_n)]$ contributing to  $\Omega^{\otimes n} \circ \Delta^{(n)}[(\xi_1 \otimes a_1) \wedge \ldots \wedge( \xi_n \otimes a_n)] $ is given by
$$ \sum_{\sigma \in \Sb_n} (-1)^{f(\sigma, |a_1|,\ldots,|a_n|)} (\xi_{\sigma(1)} \otimes a_{\sigma(1)}) \boxtimes \ldots \boxtimes (\xi_{\sigma(n)} \otimes a_{\sigma(n)}) \,\text{.}$$
Here, $(-1)^{f(\sigma, |a_1|,\ldots,|a_n|)}$ is the sign obtained after applying $\sigma$ to a product of elements of degrees $|a_1|+1,\ldots,|a_n|+1$ in a commutative graded algebra. By a second use of Proposition~\ref{pcurv},
\begin{eqnarray*}
\Omega^n[(\xi_1 \otimes a_1) \wedge \ldots \wedge( \xi_n \otimes a_n)]&=& \sum_{\sigma} (-1)^{f(\sigma, |a_1|,\ldots,|a_n|)}  da_{\sigma(1)}\ldots da_{\sigma(n)} \otimes \xi_{\sigma(1)} \otimes \ldots \otimes \xi_{\sigma(n)}\\
 &=& \sum_{\sigma} da_1 \ldots da_n \otimes  \xi_{\sigma(1)} \otimes \ldots \otimes \xi_{\sigma(n)}\,\text{.} \end{eqnarray*}
This finishes the proof of the proposition.
\eproof

\bprop 
\la{ptel} 
For any $\, P\,\in\,I^{r+1}(\g) $, we have
\begin{equation*} 
P(\theta. [\theta, \theta]^{n-r}. \Omega^{2r-n})[(\xi_0 \otimes a_0) \wedge \ldots \wedge (\xi_n \otimes a_n)]  \,=\, \sum_{\sigma \in \Sb_{n+1}} \, \pm\, A_{\sigma,P}\ a_{\sigma(0)}\ldots a_{\sigma(2n-2r)} \,da_{\sigma(2n-2r+1)}\ldots da_{\sigma(n)} 
\end{equation*}
where 
\[
A_{\sigma,P}\,= c_{n,r}\,P(\xi_{\sigma(0)}, [\xi_{\sigma(1)}, \xi_{\sigma(2)}],\ldots,[\xi_{\sigma(2n-2r-1)}, \xi_{\sigma(2n-2r)}], \xi_{2n-2r+1}, \ldots, \xi_{\sigma(n)}) \,\text{.}
\]
Here $c_{n,r}$ is a nonzero constant depending only on $n$ and $r$, with $ c_{n,n} = 1 $ 
\mbox{\rm ;} 
the sign $\, \pm \,$ in the sum is obtained by applying $\sigma$ to a product of elements of degrees $|a_0|+1,\ldots,|a_n|+1$ in a commutative graded algebra.
\eprop

\begin{remark} 
The formula of Proposition~\ref{ptel} appears in \cite{Te} as an {\it ad hoc} definition 
(see {\it op.cit.}, (2.2)). Proposition~\ref{ptel} thus explains the origin of this formula and clarifies the computations of \cite{Te}.
\end{remark}

\begin{proof}  Let $\beta_i:= \xi_i \otimes a_i$ for brevity. Indeed, the component of
$\Delta^{(n+1)} [(\xi_0 \otimes a_0) \wedge \ldots \wedge (\xi_n \otimes a_n)] $
in $ \g(A) \boxtimes \wedge^2 \g(A)^{\boxtimes n-r} \boxtimes \g(A)^{\boxtimes 2r-n} $
is given by
$$C'. \sum_{\sigma \in \Sb_{n+1}} \pm \beta_{\sigma(0)} \boxtimes \beta_{\sigma(1)} \wedge \beta_{\sigma(2)} \boxtimes \ldots \boxtimes \beta_{\sigma(2n-2r-1)} \wedge \beta_{\sigma(2n-2r)} \boxtimes \beta_{2n-2r+1} \boxtimes \ldots \boxtimes \beta_{\sigma(n)} $$
where $C'$ is a positive constant depending only on $n$  and $r$. By Proposition~\ref{pcurv}, the desired proposition follows (with $C = C'(-2)^{n-r}$).
\eproof

\subsection{Lie and cyclic homology}

Let $P\,\in\,I^{r+1}(\g)$. Let $\mathcal A$, $\theta$ be as in Section~\ref{conv}. Recall from~\eqref{csdecomp} that $ \TP(\theta)\,=\,\sum_{n=r}^{2r} A_{n-r}\Psi_{n-r,P}$
where $\Psi_{n-r,P}\,\in\, \Hom(\Sym^{n+1}(\g(A)[1]), \Omega^{2r-n}_A[n-2r])$.
Also recall that in Section~\ref{conv}, the original homological grading in $A$ was inverted to give a cohomological grading for $\C(\g(A);k)$ and allow $\dr(A)$ and $\mathcal A$ to have their natural cohomological gradings. Invert homological degrees once again to restore the original homological grading of $A$, and thereby $\C(\g(A);k)$. This inverts the natural cohomological gradings of $\dr(A)$ and $\mathcal A$, giving them a homological grading. Let $s$ denote the operator increasing homological degree by $1$. Then, $s^{2r}\TP(\theta) \,\in\, \oplus_{n=r}^{2r} \Hom(\Sym^{n+1}(\g(A)[1]), \Omega^{2r-n}_A[n])$. Recall that
$$ \mathrm{CC}^{(r)}(\drm(A))\,:=\, (\oplus_{n=r}^{2r} \Omega^{2r-n}_A[n], d+\delta) \,,$$
where $d\,:\,\Omega^{2r-n-1}[n-1] \rar \Omega^{2r-n}[n]$ is viewed as a differential with homological degree $-1$ that vanishes when $n=r$. It follows that $s^{2r}\TP(\theta)$ gives a map of degree $-1$ from $\CE(\g(A);k)$ to
$\mathrm{CC}^{(r)}(\drm(A))$.
\bthm \la{csliecyclic}
Let $A$ be a smooth commutative DG algebra. Then $s^{2r}\TP(\theta)$ induces a map on homologies $\H_{\bullet+1}(\g(A);k) \rar \rHC_{\bullet}^{(r)}(A)$.
\ethm
%

%

\bproof
By Proposition~\ref{csforms}, $d(\TP(\theta))\,=\, P(\Omega^{r+1})\, \text{.}$
Hence,
$$ \sum_{n=r}^{2r} A_{n-r}[(d+\delta)\Psi_{n-r,P} + \Psi_{n-r,P}(d+\delta)] \,=\, P(\Omega^{r+1}) \,\text{.}$$
Comparing the components of both sides in $\Hom(\Sym^{r+1}(\g(A)[1]), \Omega^{r+1}_A[-r-1])$, we see that
\begin{flalign}
&d \Psi_{0,P} \,=\, P(\Omega^{r+1})&\\
\la{trunc} &\delta\Psi_{0,P}+ \sum_{n=r+1}^{2r} A_{n-r}[(d+\delta)\Psi_{n-r,P} + \Psi_{n-r,P}(d+\delta)] \,=\, 0&
\end{flalign}
We now note that the left hand side of~\eqref{trunc} is exactly $ds^{2r}\TP(\theta)$, provided $s^{2r}\TP(\theta)$ is viewed as a degree $-1$ element of $\Hom(\CE(\g(A);k), \mathrm{CC}^{(r)}(\drm(A)))$: indeed, the differential of $\mathrm{CC}^{(r)}(\drm(A))$ restricted to the graded subspace $\Omega^r_A[r]$ of $\mathrm{CC}^{(r)}(\drm(A))$ is exactly $\delta$. Hence, $s^{2r}\TP(\theta)$ gives a map of complexes from $\CE(\g(A);k)[1]$ to $\mathrm{CC}^{(r)}(\drm(A))$. The induced map on homologies gives a map of graded vector spaces from $\H_{\bullet+1}(\g(A);k)$ to $\rHC^{(r)}_{\bullet}(A)$ as desired.
\eproof

Note that if $A$ is a smooth {\it augmented} commutative DGA, then the connection $\theta$ restricts to a connection on the convolution DGA $\mathbf{Hom}(\C(\g(\bar{A});k), \dr(A))$. It follows that Theorem~\ref{csliecyclic} can be modified for smooth augmented commutative DG algebras, giving
\bthm \la{csliecyclicaug}
Let $A$ be a smooth augmented commutative DG algebra. Then $s^{2r}\TP(\theta)$ induces a map on homologies $\H_{\bullet+1}(\g(\bar{A});k) \rar \rHC_{\bullet}^{(r)}(A)$.
\ethm
Next, we compare $s^{2r}\TP(\theta)$ with another construction of a map from $\H_{\bullet+1}(\g(\bar{A});k)$ to $\rHC_{\bullet}^{(r)}(A)$ that we give below.

\subsection{From Lie to cyclic homology: the second construction}
Let $A$ be an augmented commutative DG algebra. There is a more direct construction of a map from the Lie homology of $\g(\bar{A})$ to the (shifted) reduced cyclic homology of $A$, which works even if $A$ is not smooth. This construction is dual the construction of the Drinfeld traces recalled from~\cite{BFPRW} in Section~\ref{secdrintrace}. We begin with the map
$$ \g(\bar{A})[1] \,\cong\, \bar{A}[1] \otimes \g\,,\,\,\,\, \xi \otimes a \mapsto a \otimes \xi \,\text{.}$$
Precomposing this with the natural projection $\bSym^c(\g(\bar{A})[1]) \twoheadrightarrow \g(\bar{A})[1]$, we obtain map of graded vector spaces
$$\bSym^c(\g(\bar{A})[1]) \rar \bar{A}[1]\otimes \g\,,$$
which is equivalent to a map of graded vector spaces
\begin{equation} \la{cogen} \bSym^c(\g(\bar{A})[1]) \otimes \g^{\ast} \rar \bar{A}[1] \,\text{.}\end{equation}
Note that the Lie coalgebra structure on $\g^{\ast}$ together with the cocommutative (conilpotent) DG coalgebra structure on $\bSym^c(\g(\bar{A})[1])$ makes
$\bSym^c(\g(\bar{A})[1]) \otimes \g^{\ast}$ a (conilpotent) graded Lie coalgebra. It follows that~\eqref{cogen} cogenerates a morphism of graded Lie coalgebras
\begin{equation} \la{coprim} \bSym^c(\g(\bar{A})[1]) \otimes \g^{\ast} \rar \mathcal L^c(\bar{A}[1])\,, \end{equation}
where $\mathcal L^c(W)$ denotes the cofree Lie coalgebra cogenerated by a graded vector space $W$. Equip $\bSym^c(\g(\bar{A})[1])$ with the (co)differential in $\C(\g(\bar{A});k)$ and equip $\mathcal L^c(\bar{A}[1])$ with differential in $\bB_{\mathtt{comm}}(\bar{A})$ (see~\cite[Section 6.2.1]{BFPRW}). Then we have

\blemma
The map~\eqref{coprim} is a map of (conilpotent) DG Lie coalgebras.
\elemma

The map~\eqref{coprim} therefore induces a map of cocommutative DG coalgebras
\begin{equation} \la{symcoalg}  \C(\g(\bar{A});k) \otimes \bSym^c(\g^{\ast}) \rar \bSym^c(\mathcal L^c(\bar{A}[1])) \,\text{.}\end{equation}
There is an isomorphism of complexes $ \Sym^{\ast}\,:\, T^c(\bar{A}[1]) \,\cong\, \bSym^c(\mathcal L^c(\bar{A}[1]))$ dual to the symmetrization map  (where $T^c(\bar{A}[1])$ is equipped with the bar differential). We therefore, obtain a map of complexes
\begin{equation} \la{cetobar} \varphi(\mbox{--},\mbox{--})\,:\, \C(\g(\bar{A});k) \otimes \bSym^c(\g^{\ast}) \rar \bB(A) \,\text{.}\end{equation}

Let $\bB(A)^{\natural}$ denote the cocommutator subcomplex of $\bB(A)$. Let $\bSym^{r+1}(\mathcal L^c(\bar{A}[1]))$ continue to denote the image of $\bSym^{r+1}(\mathcal L^c(\bar{A}[1]))$ in $\bB(A)$ under the inverse of the isomorphism $\Sym^{\ast}$. Then, for a polynomial $P\,\in\, I^{r+1}(\g)$,
\bprop
$\varphi(\mbox{--}, P)$ gives a map of complexes
$$ \varphi_P\,:\, \CE(\g(\bar{A});k)\rar \bSym^{r+1}(\mathcal L^c(\bar{A}[1])) \cap \bB(A)^{\natural} \,\cong\, \overline{\mathrm{C}}^{\lambda,(r)}(A)[1]\,\text{.}$$
\eprop

Let $\mathrm{p}\,:\, \mathrm{CC}^{(r)}(\drm(A))) \rar \Omega^r_{\bar{A}}/d\Omega^{r-1}_{\bar{A}}[r]$ be as in~\eqref{pcyciso}. Let $\varepsilon\,:\, \Omega^r_{\bar{A}}/d\Omega^{r-1}_{\bar{A}}[r] \rar \overline{\mathrm{C}}^{\lambda,(r)}(A)$ be as in ~\eqref{epscyclic} and let $\pi_r\,:\, \overline{\mathrm{C}}^{\lambda,(r)}(A) \rar \Omega^r_{\bar{A}}/d\Omega^{r-1}_{\bar{A}}[r] $ be as in~\eqref{picyciso}. For notational brevity, use the same symbol to denote a map of complexes $X \rar Y$ and the induced map $X[i] \rar Y[i]$ for any $i$. Then,

 \bthm \la{csandabscomp}
 Let $A=(\Sym(W), \delta)$ be augmented over $k$. Then, the following diagram commutes in $\mathscr{D}(k)$:
 $$\begin{diagram}
 \CE(\g(\bar{A});k) & & \\
   \dTo^{\frac{1}{(r+1)!}s^{2r}\TP(\theta)} &  \rdTo^{\varphi_P}& \\
 \mathrm{CC}^{(r)}(\drm(A)))[1] & \rTo^{\varepsilon \circ \mathrm{p}} & \overline{\mathrm{C}}^{\lambda,(r)}(A)[1]
 \end{diagram}$$
 \ethm
\bproof
We begin the proof with the following proposition.
\bprop \la{pderhamcs}
$$ \mathrm{p} \circ s^{2r}\TP(\theta)\,=\,s^{2r}P(\theta.\Omega^r)$$
\eprop
\bproof
By~\eqref{csdecomp}, $\TP(\theta)\,=\,\sum_{n=r}^{2r} A_{n-r}P(\theta.[\theta,\theta]^{n-r}\Omega^{2r-n})$. Note that when $s^r\TP(\theta)$ is interpreted as a map of graded vector spaces from
$\C(\g(\bar{A});k)$ to $\mathrm{CC}^{(r)}(\drm(A))$, the image of $P(\theta.[\theta,\theta]^{n-r}\Omega^{2r-n})$ lies in $\Omega^{2r-n}[n]$. It follows that $\mathrm{p} \circ P(\theta.[\theta,\theta]^{n-r}\Omega^{2r-n})\,=\,0$ for $n>r$. Since $A_0=1$, the desired proposition follows.
\eproof
By Proposition~\ref{pderhamcs}, it suffices to verify that in $\mathscr{D}(k)$,
\begin{eqnarray} \la{compcshkr} \varphi_P &=& \frac{1}{(r+1)!} \varepsilon \circ s^{2r}P(\theta.\Omega^r)\,\text{.} \end{eqnarray}
Further, since $I_{\rm HKR}\,:\,  \overline{\mathrm{C}}^{\lambda,(r)}(A)[1] \rar \Omega^r_{\bar{A}}/d\Omega^{r-1}_{\bar{A}}[r+1]$ is a quasi-isomorphism inverting $\varepsilon$, it suffices to verify that
\begin{equation} \la{compcshkr0.5} I_{\rm HKR} \circ \varphi \,=\, \frac{1}{(r+1)!} s^{2r}P(\theta.\Omega^r)\,\text{.} \end{equation}

Note that both sides of~\eqref{compcshkr0.5} are honest maps of complexes.  By Proposition~\ref{pcurv2},  $s^{2r}P(\theta.\Omega^r)$ vanishes on all chains in $\C(\g(\bar{A});k)$ that are not in $\bSym^{r+1}(\g(\bar{A})[1])$. Similarly, $\varphi_P$ vanishes on chains in $\bSym^n(\g(\bar{A})[1])$ for any $n<r+1$. $\varphi_P$ maps chains that are in $\bSym^n(\g(\bar{A})[1])$, $n>r+1$ to chains in
$\bSym^{r+1}(\mathcal L^c(\bar{A}[1])) \cap \bB(A)^{\natural} $ that are a linear combination of summands with at least one factor in
$\mathcal L^{c, \geq 2}(\bar{A}[1])$. Hence, $I_{\rm HKR} \circ \varphi_P$ vanishes on chains that are not in $\bSym^{r+1}(\g(\bar{A})[1])$.

On chains that are in $\bSym^{r+1}(\g(A)[1])$, Proposition~\ref{ptel} gives\footnote{For the rest of this proof, the sign $\pm$ in front of each summand is the sign obtained after applying $\sigma$ to a product of elements of degrees $|a_0|+1,\ldots,|a_n|+1$ in a commutative graded algebra.}:
 $$ P(\theta.\Omega^r)((\xi_0 \otimes a_0) \wedge \ldots \wedge (\xi_r \otimes a_r))\,=\, \sum_{\sigma \in \Sb_{r+1}}  \pm a_{\sigma(0)}da_{\sigma(1)} \ldots da_{\sigma(r)} P(\xi_{\sigma(0)},\ldots, \xi_{\sigma(r)})\,\text{.}$$

On the other hand, by~\eqref{symcoalg},
$$ [(\xi_0 \otimes a_0) \wedge \ldots \wedge (\xi_r \otimes a_r)] \otimes P \mapsto a_0 \wedge \ldots \wedge a_r P(\xi_0 , \ldots , \xi_i) \,\text{.}$$
Thus,
$$\varphi(\mbox{--},P)\,:\, (\xi_0 \otimes a_0) \wedge \ldots \wedge (\xi_r \otimes a_r) {\mapsto} \sum_{\sigma \in \Sb_{i+1}} \pm a_{\sigma(0)} \otimes \ldots \otimes a_{\sigma(r)} P(\xi_{\sigma(0)},\ldots,\xi_{\sigma(r)}) \,\text{.}$$
Recall from~\cite[Section 1.3]{Q} that the identification between $\overline{\mathrm{C}}^{\lambda}(A)[1] \,\cong\, \bB(A)^{\natural}$ is given by the operator $N$ which acts on $A[1]^{\otimes n}$ by $1+\tau+ \ldots+\tau^{n-1}$ where $\tau$ denotes the $n$-cycle $(0,1,\ldots, n-1)$. It follows that
$$N^{-1}(\sum_{\sigma \in \Sb_{r+1}} \pm a_{\sigma(0)} \otimes \ldots a_{\sigma(r)} P(\xi_{\sigma(0)},\ldots,\xi_{\sigma(r)})) \,=\, \frac{1}{r+1} \sum_{\sigma \in \Sb_{r+1}} \pm  a_{\sigma(0)} \otimes \ldots \otimes a_{\sigma(r)} P(\xi_{\sigma(0)},\ldots,\xi_{\sigma(r)}))\,\text{.}$$
Hence,
$$ \varphi_P((\xi_0 \otimes a_0) \wedge \ldots \wedge (\xi_r \otimes a_r))\,=\,\frac{1}{r+1} \sum_{\sigma \in \Sb_{r+1}} \pm  a_{\sigma(0)} \otimes \ldots \otimes a_{\sigma(r)} P(\xi_{\sigma(0)},\ldots,\xi_{\sigma(r)})) \,\text{.}$$
Therefore,
$$ I_{\rm HKR} \circ \varphi_P((\xi_0 \otimes a_0) \wedge \ldots \wedge (\xi_r \otimes a_r)) \,=\, \frac{1}{(r+1)!} \sum_{\sigma \in \Sb_{r+1}} \pm a_{\sigma(0)}da_{\sigma(1)}\ldots da_{\sigma(r)}P(\xi_{\sigma(0)},\ldots, \xi_{\sigma(r)})\,\text{.}$$
This proves the desired theorem.
\eproof

\section{The HKR and co-HKR maps} \la{HKRcoHKR}

In this Appendix, we present a concise proof of Theorem~\ref{conj1}. As explained in Section~\ref{directtraceformula}, the method used for this proof can also be used to give a direct formula for traces of symmetric algebras from which Theorem~\ref{trdiffop} can be derived.

\subsection{Recollections and notation}
Let $W$ be a finite-dimensional $k$-vector space. As usual, let $A =\Sym(W)$ and let $C=\Sym^c(W[1])$ be the coalgebra Koszul dual to $A$. Fixing a basis $\{x_1,\ldots,x_N\}$ of $W$, we identify $A$ with the polynomial algebra $A=k[x_1,\dots, x_N]$. Let $dx_j\,:=\,\twbs x_i$, so that $C$ is identified with the polynomial coalgebra $C\cong k[dx_1,\dots, dx_N]$, equipped with the un-shuffle coproduct. Hence,
\[
 \drm(A)\cong \drm(C)\cong k[x_1,\dots,x_N,dx_1,\dots, dx_N]/k\, \text{.}
\]
The de Rham differential then is defined in the obvious manner by setting $d(x_j)=dx_j$.
Furthermore, we denote by $R=\cb(C)$ the minimal quasi free resolution of $A$ as before.
Concretely, $R$ is free as a graded algebra generated by symbols
\[
x_I := s^{-1} (\prod_{i\in I} dx_i )
\]
for $I\subset \{1,\dots,N\}$, where the product is taken in lexicographic order to fix the sign.
It will be convenient for us to extend this notation also to multisets, i.e., sets with multiplicities, by declaring that $x_I=0$ if the multiset $I$ contains any symbols with multiplicity greater than one.

Furthermore, given arbitrary elements $u_1,\dots, u_M\in W$ we will use similar notation and write
\begin{align*}
u_I&:= s^{-1} (\prod_{i\in I} du_i ) & &\text{and} & du_I&:= \prod_{i\in I} du_i
\end{align*}
where in both cases the product is taken in the lexicographic order.


\subsection{Remark on the map $T$}
Let $A_1$ and $A_2$ be dg (or $A_\infty$) algebras. Recall from~\cite[Sec. 4]{Ke} that an $A_{\infty}$-morphism $f\,:\,A_1 \rar A_2$ is equivalent to a twisting cochain $f\,:\,\bB(A_1) \rar 
A_2$, where $\bB(A_1)$ denotes the bar construction of $A_1$. A twisting coahain $f$ induces a map  between cyclic bar complexes
\[
\phi_f:\mathrm{C}^\lambda(A_1)\rar \mathrm{C}^\lambda(A_2)
\]
given by the formula (using cyclic indexing, i.e., $-1 \equiv n$ etc.)
\[
(a_0,\dots, a_n) \mapsto \sum_{r=1}^n \sum_{j_1+\dots+j_r=n+1} \sum_{i=1}^{j_1}
\pm (f_{j_1}(a_{1-i}, \dots, a_{j_1-i}), f_{j_2}(a_{j_1-i+1},\dots, a_{j_1+j_2-i}),\dots, f_{j_r}(\dots, a_{-i})).
\]
This formula is compatible with composition of $A_\infty$ morphisms, making $\mathrm{C}^\lambda(\mbox{--})$  a functor from the category of dg algebras with $A_\infty$ morphisms to the category of chain complexes. If we require in addition that the $A_\infty$ morphisms be unital, i.e., $f_1(1_{A_1})=1_{A_2}$ and $f_n(\dots,1_{A_1},\dots)=0$ for $n>1$, then $\phi_f$ descends to a map
 \[
 \phi_f\,:\, \overline{\mathrm{C}}^{\lambda}(A_1) \rar  \overline{\mathrm{C}}^{\lambda}(A_2)
 \]
 between reduced cyclic complexes.

In particular, let $A_1=A$ and let $A_2=R$, and let $\theta: \bB(A) \to R$ be the twisting cochain corresponding to a (unital) $A_\infty$ right inverse to the canonical projection $R\to A$.
The composite map
\[
\begin{diagram}  \overline{\mathrm{C}}^{\lambda}(A) & \rTo^{\phi_{\theta}} &  \overline{\mathrm{C}}^{\lambda}(R) & \rOnto & R_{\n} \end{diagram}
\]
is precisely the map $T$ from~\eqref{bkr422}. Functoriality of $ \overline{\mathrm{C}}^{\lambda}(\mbox{--})$ immediately implies the following lemma.

\begin{lemma}\label{lem:stau}
The map $T : \overline{\mathrm{C}}^{\lambda}(A)\to R_\natural$ induces the same map in homology as the zigzag of quasi-isomorphisms
\[
\begin{diagram} \overline{\mathrm{C}}^{\lambda}(A) & \lOnto &  \overline{\mathrm{C}}^{\lambda}(R) & \rOnto & R_\natural \end{diagram} \,\text{.}
\]
\end{lemma}

The Lemma provides us with an alternative definition of $T$ without $A_\infty$ morphisms, and hence with another equivalent formulation of Theorem~\ref{conj1}.

\subsection{Proof of Theorem~\ref{conj1}}
We have the following diagram of quasi-isomorphisms:
\begin{equation}
\label{equ:zigzag}
\begin{diagram} \drm(A)/d\drm(A) & \lTo^{I_{\rm HKR}} & \overline{\mathrm{C}}^{\lambda}(A) & \lTo &
\overline{\mathrm{C}}^{\lambda}(R)  & \rTo& R_\natural \,=\, \overline{\mathrm{C}}^{\lambda}(C)[1]  & \rTo^{\varepsilon} &
\ker(d_C)[1] \end{diagram}
\end{equation}

Here, $\varepsilon$ is the quasi-isomorphism inverting the co-HKR map (see~\eqref{epscyccoalg}). Using Lemma \ref{lem:stau} the statement of Theorem~\ref{conj1} is equivalent to the assertion that the isomorphism in homology induced by~\eqref{equ:zigzag} above is just the de Rham differential.
To check this, let us start with the element of $\drm(A)/d\drm(A) $ represented by $u_1u_2\cdots u_n du_{n+1}\cdots du_{n+p}$ for elements $u_1,\dots,u_{n+p}\in W$.

\begin{lemma}
A representative in $\overline{\mathrm{C}}^{\lambda}(R)$ of the homology class corresponding to $\alpha\,:=\,u_1u_2\cdots u_n du_{n+1}\cdots du_{n+p}$ is given by
\begin{multline}\label{equ:betadef}
\beta=
\frac{1}{n!}
\sum_{\sigma\in \Sb_n}
 \sum_{m= 0}^p
  \sum_f (-1)^f
 u_{\{\sigma(1)\}\cup f^{-1}(1)} u_{\{\sigma(2)\}\cup f^{-1}(2)} \cdots u_{\{\sigma(n)\}\cup f^{-1}(n)}
  \otimes u_{f^{-1}(\bar 1)}\otimes u_{f^{-1}(\bar 2)}\otimes \cdots \otimes u_{f^{-1}(\bar m)},
 \end{multline}
where the sum is over maps $f:\{n+1,\dots,n+p\}\to \{1,\dots,n,\bar 1,\dots, \bar m\}$ such that each $\bar j$ is hit at least once. The sign depends only on $f$, and is determined by the permutation of the factors $du_j$ (which are counted as odd) appearing in a term in the above formula.
\end{lemma}
\begin{proof}
Let us first verify that the stated element indeed maps to $\alpha$ under the composition of the two left-most maps in \eqref{equ:zigzag}. Indeed, the only elements that survive the projection $R\to A$ are those $u_I$'s with $|I|=1$. In other words, only the top summand $m=p$ contributes, and for $m=p$ the only allowed maps $f$ are permutations of $\{\bar 1,\dots, \bar m\}$. There are $p!$ such maps, canceling the prefactor in the definition of the Hochschild-Kostant-Rosenberg morphism. The symmetrization $\frac 1 {n!}\sum_{\sigma\in \Sb_n}$ is inessential and can be omitted before or after mapping via the HKR morphism. Hence we exactly recover $\alpha$ as desired.

Next we have to show that $\beta$ is a cocycle. We claim that, in fact, each $\sigma$-summand in \eqref{equ:betadef} is separately closed. This is verified by a straightforward, but tedious computation. We sketch this computation for the term where $\sigma$ is the identity. Abbreviating the first tensor factor as $X$ to save space, the Hochschild differential of this term is
\begin{eqnarray*}
\lefteqn{\sum_{m=0}^p\,
  \sum_f \,(-1)^f
  \delta\left(
 X \otimes u_{f^{-1}(\bar 1)}\otimes u_{f^{-1}(\bar 2)}\otimes \cdots \otimes u_{f^{-1}(\bar m)}\right) =} \\
& & \sum_{m= 0}^p\, \sum_f\, (-1)^f
\bigl(\pm X u_{f^{-1}(\bar 1)}\otimes u_{f^{-1}(\bar 2)}\otimes \cdots \otimes u_{f^{-1}(\bar m)} 
\pm X\otimes  u_{f^{-1}(\bar 1)} u_{f^{-1}(\bar 2)}\otimes \cdots \otimes u_{f^{-1}(\bar m)}
 +\ \ldots \\*[2ex]
& & \pm X\otimes  u_{f^{-1}(\bar 1)}\otimes  u_{f^{-1}(\bar 2)}\otimes \cdots \otimes  u_{f^{-1}(\overline{ m-1})} u_{f^{-1}(\bar m)}
\pm
u_{f^{-1}(\bar m)} X\otimes  u_{f^{-1}(\bar 1)}\otimes  u_{f^{-1}(\bar 2)}\otimes \cdots \otimes  u_{f^{-1}(\overline{m-1})}
\bigr)\, .
\end{eqnarray*}
Note that the terms $u_{f^{-1}(\bar 1)} u_{f^{-1}(\bar 2)}$ appearing in the above summands for some $m$ reproduce $\delta_R u_{f^{-1}(\bar 1)}$ in the summand for $m-1$, where $\delta_R$ is the differential on $R$. It can alaso be seen without difficulty that the sign before the above summands with $u_{f^{-1}(\bar{1})}u_{f^{-1}(\bar{2})}$ coming from $\delta_R$ is precisely $(-1)^{f+|X|+|u_{f^{-1}(\bar{1})}|+2}$ while the sign on the summands with $u_{f^{-1}(\bar{1})}u_{f^{-1}(\bar{2})}$ coming from the Hochschild differential $\delta$ is $(-1)^{f+|X|+|u_{f^{-1}(\bar{1})}|+1}$. Thus, the above summands with $u_{f^{-1}(\bar{1})}u_{f^{-1}(\bar{2})}$ coming from the Hochschild differential cancel out similar summands from $\delta_R$ applied to $f$ corresponding to $m-1$. A similar statement holds, of course, if we replace 1 and 2 by $i$ and $j$. Furthermore, by a similar argument, the terms $X\otimes  u_{f^{-1}(\bar 1)}$ and $u_{f^{-1}(\bar m)} X$ appearing in some summand $m$ cancel out the terms yielding $\delta_RX$ for one lower $m$. This shows that \eqref{equ:betadef} is indeed closed under the total differential.
\end{proof}

Under the map $\overline{\mathrm{C}}^{\lambda}(R) \to R_\natural$ the element $\beta$ above is sent   to
\begin{equation}\label{equ:Rnatelement}
\frac{1}{n!}
\sum_{\sigma\in \Sb_n} \sum_f (-1)^f
  u_{\{\sigma(1)\}\cup f^{-1}(1)} u_{\{\sigma(2)\}\cup f^{-1}(2)} \cdots u_{\{\sigma(n)\}\cup f^{-1}(n)}.
\end{equation}
Here the sum is over maps $f:\{n+1,\dots,n+p\}\to \{1,\dots, n\}$.
The corresponding element in $\overline{\mathrm{C}}^{\lambda}(C)[1]$ is obtained by replacing $u_I$'s by the corresponding $du_I$'s, and putting tensor signs between factors. This gives
\begin{equation}\label{equ:topcomposition}
\frac{1}{n!}
\sum_{\sigma\in \Sb_n}
  \sum_f (-1)^f
  du_{\{\sigma(1)\}\cup f^{-1}(1)} \otimes du_{\{\sigma(2)\}\cup f^{-1}(2)} \otimes \cdots \otimes du_{\{\sigma(n)\}\cup f^{-1}(n)}.
\end{equation}
We need to compute the image of the above element under the map $\varepsilon$ (see~\eqref{epscyccoalg}). The map dual to $\varepsilon$ has been explicitly described in~\eqref{epscyclic}. In particular,  all terms that have more than one non-linear tensor factor are sent to $0$ under $\varepsilon$. Hence, in the above sum over $f$ we retain only maps $f:\{n+1,\dots,n+p\}\to \{1,\dots, n\}$ that have a single element in the image. We will assume that $p>0$, leaving the simpler case of $p=0$ to the reader. Then the summands of~\eqref{equ:topcomposition} contributing nontrivially add up to
\[
\frac{1}{n!}
\sum_{\sigma\in \Sb_n}
 \sum_{i=1}^n
  du_{\sigma(1)} \otimes \cdots \otimes du_{\sigma(i)} du_{\{n+1,\dots,n+p \} } \otimes \cdots \otimes du_{\sigma(n)}
  =
\frac{n}{n!}
\sum_{\sigma\in \Sb_n}
  du_{\sigma(1)}du_{\{n+1,\dots,n+p \} }  \otimes du_{\sigma(2)} \otimes  \cdots \otimes du_{\sigma(n)}
  .
\]

One can now easily verify that $
\varepsilon[\eqref{equ:topcomposition}] \,=\, \sum_{i=1}^n  du_i du_{n+1}\cdots du_{n+p} u_1\cdots \hat u_i \cdots u_n=  d\alpha\in \ker d\,\text{.}
$
This proves Theorem~\ref{conj1}.
\hfill\qed

\subsection{Simplified trace formula} \la{simpletraceform}
Note that we can get a formula for the trace map by just mapping the element \eqref{equ:Rnatelement} of $R_\natural$ to the abelianization $R_{\ab}\cong k[x_I\mid I\subset \{1,\dots ,N\}, I\neq\emptyset]$.

\begin{theorem}\label{thm:simpletrace}
The trace map $ \TTr(A):\drm(A)/d\drm(A) \to R_{\ab}$ satisfies the formula
\[
\TTr(A)(u_1u_2\cdots u_n du_{n+1}\cdots du_{n+p})
=
 \sum_f (-1)^f
  u_{\{1\}\cup f^{-1}(1)} u_{\{2\}\cup f^{-1}(2)} \cdots u_{\{n\}\cup f^{-1}(n)}.
\]
where the sum is over maps $f:\{n+1,\dots,n+p\}\to \{1,\dots,n\}$.
\end{theorem}

\subsection{Another proof of Theorem~\ref{trdiffop}} \la{directtraceformula}

For $\omega\,\in \, \Sym^{n}(W) \otimes \wedge^p(W)$, let $F(\omega)\,:=\,\frac{1}{(n+1)!}[\theta.\Omega^{n}](\omega)$. Explicitly, $(n+1)!F(\omega)$ is given by the right hand side of~\eqref{tracegen} (for $q=n$).  We can directly verify that for $\omega\,\in\,\Sym^n(W) \otimes \wedge^p(W)$,  the right hand side of the formula in Theorem~\ref{thm:simpletrace} (applied to $\omega$) coincides with $F(d\omega)$. This gives us another route to the proof of Theorem~\ref{trdiffop}. Indeed,
\[
F(u_1u_2\cdots u_n du_{n+1}\cdots du_{n+p})
=
  \frac{1}{n+1}\sum_f \pm
 u_{f^{-1}(0)} u_{\{1\}\cup f^{-1}(1)} u_{\{2\}\cup f^{-1}(2)} \cdots u_{\{n\}\cup f^{-1}(n)}
\]
and where the sum is over all maps $f:\{n+1,\dots, n+p\} \to \{0,1,\dots, n\}$. Computing $F(d\omega)$ for $\omega=u_1u_2\cdots u_n du_{n+1}\cdots du_{n+p}$ we obtain
\begin{align*}
F(d\omega) \,=\,F\left(\sum_{i=1}^n u_1\cdots \hat u_i\cdots u_n du_i du_{n+1}\cdots du_{n+p} \right)
\,=\,
\frac{1}{n} \sum_{i=1}^n
\sum_f \pm
u_{f^{-1}(0)} u_{\{1\}\cup f^{-1}(1)}  \cdots \hat u_i \cdots u_{\{n\}\cup f^{-1}(n)}
\end{align*}
where the second sum in the second line is over maps $f:\{i, n+1,\dots, n+p\} \to \{0,1,\dots,\hat i,\dots, n\}$.
We will decompose this sum according to $j:=f(i)$. In particular, we split off the $j=0$-piece. This yields
\begin{align*}
F(d\omega) &=
\frac{1}{n} \sum_{i=1}^n \sum_f \pm u_{\{i\}\cup f^{-1}(0)} u_{\{1\}\cup f^{-1}(1)}  \cdots \hat u_i \cdots u_{\{n\}\cup f^{-1}(n)}\\
&\quad+
\frac{1}{n} \sum_{\substack{i,j=1 \\ i\neq j}}^n \sum_f \pm
u_{f^{-1}(0)} u_{\{1\}\cup f^{-1}(1)}  \cdots \hat u_i \cdots u_{\{j,i\} \cup f^{-1}(j)}\cdots u_{\{n\}\cup f^{-1}(n)}
\end{align*}
where the sums are now over maps $f:\{ n+1,\dots, n+p\} \to \{0,1,\dots,\hat i,\dots, n\}$. Note that the second line is zero by antisymmetry of the summand in $i$ and $j$. The first line remains and can be re-written as
\[
F(d\omega) =
\sum_f \pm
  u_{\{1\}\cup f^{-1}(1)} u_{\{2\}\cup f^{-1}(2)} \cdots u_{\{n\}\cup f^{-1}(n)}
\]
which agrees with the formula of Theorem~\ref{thm:simpletrace}.
\hfill\qed

\end{document}